\documentclass[12pt]{amsart}
\usepackage[latin1]{inputenc}
\usepackage[english]{babel}
\usepackage{amsmath}
\usepackage{amsthm}
\usepackage{amssymb}
\usepackage{amsmath, amsthm, amsfonts, mathrsfs, amsfonts}
\usepackage{latexsym}
\usepackage{textcomp}
\usepackage{enumerate}

\usepackage{xcolor}
\usepackage{hyperref}

\theoremstyle{definition}
\newtheorem{definition}{Definition}[section]
\newtheorem{notation}[definition]{Notation}

\newtheorem{remark}[definition]{Remark}

\theoremstyle{remark}

\theoremstyle{plain}

\theoremstyle{plain}
\newtheorem{cor}[definition]{Corollary}

\theoremstyle{plain}
\newtheorem{lemma}[definition]{Lemma}

\theoremstyle{plain}
\newtheorem{theorem}[definition]{Theorem}

\renewcommand{\phi}{\varphi}

\newcommand{\One}{\operatorname{\mathbf{1}}}

\theoremstyle{remark}

\theoremstyle{definition}
\newtheorem{hypo}[definition]{Hypothesis}

\newcommand{\C}{\mathrm{C}}
\newcommand{\D}{\bold{D}}

\newcommand{\F}{\mathcal{F}}
\newcommand{\G}{\mathrm{G}}
\newcommand{\M}{\mathcal{M}}

\newcommand{\N}{\mathcal{N}}
\newcommand{\X}{\mathcal{X}}
\newcommand{\Y}{\mathcal{Y}}
\newcommand{\rN}{\mathrm{N}}

\newcommand{\W}{\bold{W}}

\newcommand{\Aut}{\mathrm{Aut}}

\newcommand{\Sym}{\mathrm{Sym}}

\newcommand{\norm}{\mathrel{\unlhd}}

\def \L {\mathcal L}
\def \H {\mathcal H}

\def \ov {\overline}

\author{Valentina Grazian  and Ellen Henke }
\title{Semidirect products and wreath products of localities}

\begin{document}
\maketitle

\begin{abstract}
We develop a theory of semidirect products of partial groups and loca-lities. Our concepts generalize the notions of direct products of partial groups and localities, and of semidirect products of groups.
\end{abstract}


\section{Introduction}

Partial groups and localities were introduced by Chermak \cite{Ch} in connection with his proof of the existence and uniqueness of centric linking systems, a results which ensures in particular that there is a locality attached to every saturated fusion system. Roughly speaking, a partial group is a set $\L$ together with a ``partial product'' which is only defined on certain words in $\L$, subject to certain axioms which resemble the axioms of a group. A locality is a partial group equipped with some extra structure; in particular, every locality contains a ``Sylow subgroup'' on which one can define a fusion system.

\smallskip

Since there is no suitable notion of an action of a fusion system, it is difficult to give a general definition of semidirect products of fusion systems. However, a quite canonical definition of direct products was already introduced by Broto, Levi and Oliver \cite{BLO2}. In this paper, we approach the problem by studying semidirect products of partial groups and localities. As for groups, we define both internal and external semidirect products. After some preliminaries, we start by introducing external and internal semidirect products of partial groups in Sections~\ref{S:ExternalPartialGroup} and \ref{S:InternalPartialGroup}. Building on that, we introduce semidirect products of localities  in Section~\ref{S:SemidirectProdLoc}; as a special case, we define in Subsection~\ref{Ss:SemidirectProdGroupLoc} semidirect products of groups with localities. This is in particular useful for defining locality versions of wreath products in Section~\ref{S:WreathProducts}.

\smallskip


The concepts and results presented in this paper generalize the notions of direct products of partial groups and localities introduced by the second author of this paper \cite{direct.prod}. Forming direct products of localities is in a certain sense compatible with forming direct products of fusion systems (cf. \cite[Lemma~5.1]{direct.prod}). Semidirect pro-ducts of partial groups can also be seen as generalizations of semidirect products of groups.

\smallskip

\textit{Notation:} Let $p$ always be a prime. Throughout, we write homomorphisms of groups or of partial groups exponentially on the right hand side. 

\section{Partial groups and localities}

In this section we introduce some basic definitions and notations that will be used throughout this paper.  We refer the reader to \cite{loc1} for a more comprehensive introduction to partial groups and localities.

\subsection{Partial groups} 
For any set $\M$, write $\W(\M)$ for the free monoid on $\M$. Thus, an element of $\W(\M)$ is a word with letters in $\M$. The multiplication on $\W(\M)$ consists of concatenation of words, to be denoted $u \circ v$. The empty word will be denoted by $\emptyset$. For any word $u\in\W(\M)$ and $k\in\mathbb{N}_0$, $u^k$ stands for the concatenation of $k$ copies of $u$, i.e. $u^k$ is defined inductively by $u^0=\emptyset$ and $u^{k+1}=u\circ u^k$ for every $k\in\mathbb{N}_0$. 


\begin{definition}[Partial Group]\label{partial}
Let $\L$ be a non-empty set, let $\D(\L)$ be a subset of $\W(\L)$, let $\Pi \colon  \D(\L) \rightarrow \L$ be a map and let $(-)^{-1} \colon \L \rightarrow \L$ be an involutory bijection, which we extend to a map 
\[(-)^{-1} \colon \W(\L) \rightarrow \W(\L), w = (g_1, \dots, g_k) \mapsto w^{-1} = (g_k^{-1}, \dots, g_1^{-1})).\]
We say that $\L$ is a \emph{partial group} with product $\Pi$ and inversion $(-)^{-1}$ if the following hold:
\begin{itemize}
\item[(PG1)]  $\L \subseteq \D(\L)$  (i.e. $\D(\L)$ contains all words of length 1), and
\[  u \circ v \in \D(\L) \Rightarrow u,v \in \D(\L);\]
(So in particular, $\emptyset\in\D(\L)$.)
\item[(PG2)] $\Pi$ restricts to the identity map on $\L$;
\item[(PG3)] $u \circ v \circ w \in \D(\L) \Rightarrow u \circ (\Pi(v)) \circ w \in \D(\L)$, and $\Pi(u \circ v \circ w) = \Pi(u \circ (\Pi(v)) \circ w)$;
\item[(PG4)] $w \in  \D(\L) \Rightarrow  w^{-1} \circ w\in \D(\L)$ and $\Pi(w^{-1} \circ w) = \One$ where $\One:=\Pi(\emptyset)$.
\end{itemize}
\end{definition}

Note that any group $G$ can be regarded as a partial group with $\D(G)=\W(G)$ by extending the ``binary'' product to a map $\Pi_G\colon\W(G)\rightarrow G,(g_1,g_2,\dots,g_n)\mapsto g_1g_2\cdots g_n$. We will always write $\Pi_\G$ for this product.

\smallskip

If $\L$ is a partial group with product $\Pi\colon\D(\L)\rightarrow\L$ and $u=(f_1,f_2,\dots,f_n)\in\D(\L)$, then we write also $f_1f_2\cdots f_n$ for $\Pi(u)$.

\begin{lemma}\label{L:PartialGroup}
If $\L$ is a partial group and $u,v\in\W(\L)$, then the following hold:
\begin{itemize}
 \item [(a)] If $u\circ v\in\D(\L)$, then we have $(\Pi(u),\Pi(v))\in\D(\L)$ and $\Pi(u\circ v)=\Pi(\Pi(u),\Pi(v))$.
 \item [(b)] If $u\circ v\in\D(\L)$, then $u^{-1}\circ u\circ v\in\D(\L)$ and $u\circ v\circ v^{-1}\in\D(\L)$. Moreover, $\Pi(u^{-1}\circ u\circ v)=\Pi(v)$ and $\Pi(u\circ v\circ v^{-1})=\Pi(u)$.
 \item [(c)] If $u\in\D(\L)$ and $u=u^{-1}$, then $u^k\in \D(\L)$ and $\Pi(u^k)=\Pi(\Pi(u)^k)$ for every $k\in\mathbb{N}_0$.
 \item [(d)] If $u\circ v\in\D(\L)$, then $u\circ (\One)\circ v\in\D(\L)$ and $\Pi(u\circ (\One)\circ v)=\Pi(u\circ v)$. More generally, if $u_1\circ \cdots \circ u_n\in\D(\L)$ and $e_0,\dots,e_n\in\W(\{\One\})$, then $w=e_0\circ u_1\circ e_1\circ\cdots \circ e_{n-1}\circ u_n\circ e_n\in\D(\L)$ and $\Pi(w)=\Pi(u_1\circ\cdots\circ u_n)$.
\item [(e)] $\Pi(w)=\One$ for every $w\in\W(\{\One\})$.
\item [(f)] For every $f\in\L$ and every $k\in\mathbb{N}_0$, we have $(f^{-1},\One,f)^k\in\D$ and $\Pi((f^{-1},\One,f)^k)=\One$. 
\item [(g)] If $u\in\D(\L)$, then $u^{-1}\in\D(\L)$ and $\Pi(u)^{-1}=\Pi(u^{-1})$. 
\item [(h)] If $u,v,w\in\W(\L)$ with $u\circ v\circ v^{-1}\circ w\in\D(\L)$, then $u\circ w\in\D(\L)$ and $\Pi(u\circ v\circ v^{-1}\circ w)=\Pi(u\circ w)$. 
\end{itemize} 
\end{lemma}

\begin{proof}
Applying axiom (PG3) twice gives (a). Notice that $(u\circ v)^{-1}=v^{-1}\circ u^{-1}$. So the first part of property (b) follows from (PG1) and (PG4). The second part of (b) follows now from (PG3) and (PG4). If $u$ is as in (c), it follows from $\emptyset \in\D(\L)$ and (b) by induction that $u^k\in\D(\L)$. So it follows from (a) by induction that $\Pi(u^k)=\Pi(\Pi(u)^k)$, and this completes the proof of (c).

\smallskip

If $u\circ v=u\circ\emptyset\circ v\in\D(\L)$, then it follows from axiom (PG3) that $u\circ (\One)\circ v=u\circ (\Pi(\emptyset))\circ v\in\D(\L)$ and $\Pi(u\circ(\One)\circ v)=\Pi(u\circ v)$. So the first part of (d) holds, and the second part follows then by induction on the length of $e_0\circ e_1\circ\cdots\circ e_n$. Thus, (d) holds. Property (e) is a special case of (d) as $\emptyset\in\D(\L)$.

\smallskip

By axioms (PG1) and (PG4), we have $(f^{-1},f)\in\D(\L)$. So by property (d), $u:=(f^{-1},\One,f)\in\D(\L)$. As $u^{-1}=u$, it follows from (c) that $u^k\in\D(\L)$ and $\Pi(u^k)=\Pi(\Pi(u)^k)=\Pi(\One^k)$. Using (e) we conclude $\Pi(u^k)=\One$. This proves (f). 

\smallskip   

Property (g) is shown in \cite[Lemma~1.4(f)]{loc1}. Let now $u,v,w\in\W(\L)$ with $u\circ v\circ v^{-1}\circ w\in\D(\L)$. If $u=\emptyset$, then (h) follows from (b). So we may assume $u\neq\emptyset$. Write $u=(g_1,\dots,g_k)$. Using (PG3),(PG4) and property (d) one sees now that $u\circ (\Pi(v\circ v^{-1}))\circ w=u\circ (\One)\circ w=(g_1,\dots,g_{k-1})\circ (g_k,\One)\circ w\in\D(\L)$ and thus $(g_1,\dots,g_{k-1})\circ (\Pi(g_k,\One))\circ w=(g_1,\dots,g_{k-1})\circ (g_k)\circ w=u\circ w\in\D(\L)$. Moreover, $\Pi(u\circ v\circ v^{-1}\circ w)=\Pi(u\circ (\Pi(v\circ v^{-1}))\circ w)=\Pi(u\circ (\One)\circ w)=\Pi((g_1,\dots,g_{k-1})\circ (g_k,\One)\circ w)=\Pi((g_1,\dots,g_{k-1})\circ (\Pi(g_k,\One))\circ w)=\Pi((g_1,\dots,g_{k-1})\circ (g_k)\circ w)=\Pi(u\circ w)$.
\end{proof}

\begin{definition}
Let $\L$ be a partial group. For every $g\in \L$ we define
\[ \D(g) = \{ x\in \L \mid (g^{-1}, x, g) \in \D(\L)\}.\]
The map $c_g \colon \D(g) \rightarrow \L$, $x \mapsto x^g = \Pi(g^{-1}, x, g)$ is the \emph{conjugation map} by $g$. If $\H$ is a subset of $\L$ and $\H \subseteq \D(g)$, then we set 
\[\H^g = \{ h^g \mid h \in \H\}.\]
Whenever we write $x^g$ (or $\H^g$), we mean implicitly that $x\in\D(g)$ (or $\H\subseteq \D(g)$, respectively). Moreover, if $\M$ and $\H$ are subsets of $\L$, we write $N_\M(\H)$ for the set of all $g\in\M$ such that $\H\subseteq\D(g)$ and $\H^g=\H$. Similarly, we write $\C_\M(\H)$ for the set of all $g\in\M$ such that $\H\subseteq\D(g)$ and $h^g=h$ for all $h\in\H$.
\end{definition}

\begin{definition}
Let $\L$ be a partial group and let $\H$ be a non-empty subset of $\L$. The subset $\H$ is a \emph{partial subgroup} of $\L$ if
\begin{enumerate}
\item $g\in \H \Longrightarrow g^{-1} \in \H$; and
\item $w \in \D(\L) \cap \W(\H) \Longrightarrow \Pi(w) \in \H$.
\end{enumerate}
If $\H$ is a partial subgroup of $\L$ with $\W(\H)\subseteq\D(\L)$, then $\H$ is called a \emph{subgroup} of $\L$. 

\smallskip

A partial subgroup $\N$ of $\L$ is called a \emph{partial normal subgroup} of $\L$ (denoted $\N \norm \L$) if 
\[ g\in \L, n \in \N \cap \D(g) \Longrightarrow n^g \in \N.\]
\end{definition}

We remark that a subgroup $\H$ of $\L$ is always a group in the usual sense with the group multiplication defined by $hg=\Pi(h,g)$ for all $h,g\in\H$.

\subsection{Localities}

Roughly speaking, localities are partial groups with some some extra structure, in particular with a ``Sylow $p$-subgroup'' and a set $\Delta$ of ``objects'' which in a sense determines the domain of the product. Crucial is the following definition.

\begin{definition}
Let $\L$ be a partial group and let $\Delta$ be a collection of subgroups of $\L$. Define $\D(\L)_\Delta$ to be the set of words $w=(g_1, \dots, g_k) \in \W(\L)$ such that there exist $P_0, \dots ,P_k \in \Delta$ with $P_{i-1} \subseteq \D(g_i)$ and $P_{i-1}^{g_i} = P_i$ for all $1 \leq  i \leq k$; if such $P_0,\dots,P_k$ are given, then we say also that $w\in\D(\L)_\Delta$ via $P_0,P_1,\dots,P_k$, or just that $w\in\D(\L)_\Delta$ via $P_0$.
\end{definition}

\begin{remark}\label{R:DDelta}
Let $\L$ be a partial group and let $\Delta$ and $\Delta^*$ be sets of subgroups of $\L$. If $\Delta\subseteq\Delta^*$, then $\D(\L)_{\Delta}\subseteq\D(\L)_{\Delta^*}$. 

\smallskip

More generally, if there exists a map $\gamma\colon \Delta\rightarrow \Delta^*$ such that, for all $P,Q\in\Delta$ and $f\in\L$, we have
\[P\subseteq \D(f)\mbox{ and }P^f=Q\Longrightarrow \gamma(P)\subseteq \D(f)\mbox{ and }\gamma(P)^f=\gamma(Q),\]
then $\D(\L)_{\Delta}\subseteq\D(\L)_{\Delta^*}$. For, if $\gamma$ is such a map and $w=(f_1,\dots,f_n)\in\D(\L)_{\Delta}$ via $P_0,P_1,\dots,P_n\in\Delta$, then $w\in\D(\L)_{\Delta^*}$ via $\gamma(P_0),\gamma(P_1),\dots,\gamma(P_n)$.  
\end{remark}

\begin{definition}\label{locality}
Let $\L$ be a finite partial group, let $S$ be a $p$-subgroup of $\L$ and let $\Delta$ be a non-empty set of subgroups of $S$. We say that $(\L, \Delta, S)$ is a \textbf{locality} if the following hold:
\begin{enumerate}
\item $S$ is maximal with respect to inclusion among the $p$-subgroups of $\L$;
\item $\D(\L) = \D(\L)_\Delta$;
\item $\Delta$ is closed under taking $\L$-conjugates and overgroups in $S$; i.e. if $P \in \Delta$ then $P^g\in \Delta$ for every $g\in \L$ with $P \subseteq \D(g)$, and $R\in\Delta$ for every $P \leq R \leq S$.
\end{enumerate}
\end{definition}

We remark that Definition \ref{locality} is the definition of a locality given by the second author \cite[Definition 5.1]{subcentric}. It is shown there that this definition is equivalent to the one given by Chermak \cite[Definition 2.8]{loc1}.

\smallskip

We will later be in a situation where we are given a partial group $\L$ with a maximal $p$-subgroup $S$, and want to show that $(\L,\Delta,S)$ is a locality for a suitable set $\Delta$ of subgroups of $S$. We will now develop some general methods for proving this. We will need the following definition.

\begin{definition}
Let $\L$ be a partial group and let $S$ be a $p$-subgroup of $\L$. For $f\in\L$ set
\[S_f:=\{x\in S\colon x\in\D(f)\mbox{ and }x^f\in S\}.\]
We say that a set $\Delta$ of subgroups of $S$ is closed under taking $\L$-conjugates in $S$ if, for every $P\in\Delta$ and every $f\in\L$ with $P\subseteq S_f$, we have $P^f\in\Delta$.
\end{definition}

If $(\L,\Delta,S)$ is a locality, then for every $f\in\L$, the subset $S_f$ is actually a subgroup of $S$. Moreover, $P^f$ is a subgroup of $S$ for every subgroup $P$ of $S$ with $P\subseteq S_f$. We warn the reader that, if $\L$ is an arbitrary partial group with a maximal $p$-subgroup $S$, then its a priori not clear that these properties hold. In the proofs of the following results we therefore need to be very careful how we argue.

\begin{lemma}\label{L:ConstructLocality0}
 Let $\L$ be a partial group with product $\Pi\colon\D(\L)\rightarrow \L$. Let $\Delta$ be a set of subgroups of $\L$ such that $\D(\L)=\D(\L)_\Delta$. Then the following hold:
\begin{itemize}
 \item [(a)] For every $P\in\Delta$, $N_\L(P)$ is a subgroup of $\L$.
 \item [(b)] If $P\in\Delta$ and $f\in\L$ with $P\subseteq S_f$ and $P^f\in\Delta$, then $N_\L(P)\subseteq \D(f)$ and $c_f\colon N_\L(P)\rightarrow N_\L(P^f)$ is an isomorphism of groups.
 \item [(c)] If $w=(f_1,\dots,f_n)\in\D(\L)$ via $P_0,P_1,\dots,P_n\in\Delta$, then $c_{\Pi(w)}=c_{f_1}\circ c_{f_2}\circ \cdots \circ c_{f_n}$ as an isomorphism from $N_\L(P_0)$ to $N_\L(P_1)$.
\end{itemize}
\end{lemma}

\begin{proof}
By \cite[Lemma~1.6(c)]{loc1}, for any $g\in\L$, the conjugation map $c_g$ is a bijection $\D(g)\rightarrow \D(g^{-1})$ with inverse map $c_{g^{-1}}$. We will use this property throughout without further reference. We will first prove the following property from which (a) and (b) follow easily:

\begin{itemize}
 \item [(*)] Let $w=(f_1,\dots,f_n)\in\D$ via $P_0,P_1,\dots,P_n\in\Delta$, and suppose $X_0,X_1,\dots,X_n$ are subgroups of $\L$ with $P_i\unlhd X_i$ for $i=0,1,\dots,n$. Assume furthermore $X_{i-1}\subseteq \D(f_i)$ and $X_{i-1}^{f_i}=X_i$ for $i=1,\dots,n$. We show that $c_{\Pi(w)}=c_{f_1}\circ c_{f_2}\circ\cdots\circ c_{f_n}$ as a map from $X_0$ to $X_n$. 
\end{itemize}

\smallskip

To prove this choose $x\in X_0$ and notice that $w^{-1}\circ (x)\circ w\in\D=\D_\Delta$ via $P_n,P_{n-1},\dots,P_0,P_0,P_1,\dots,P_n$. Moreover, by Lemma~\ref{L:PartialGroup}(g), we have $\Pi(w)^{-1}=\Pi(w^{-1})$. 
Using axiom (PG3) of a partial group several times, we get thus $(\Pi(w)^{-1},x,\Pi(w))=(\Pi(w^{-1}),x,\Pi(w))\in\D$ and $x^{c_{\Pi(w)}}=\Pi(\Pi(w^{-1}),x,\Pi(w))=\Pi(w^{-1}\circ (x)\circ w)=\Pi(f_n^{-1},\dots,f_1^{-1},x,f_1,\dots,f_n)=x^{c_{f_1}\circ c_{f_2}\circ \cdots\circ c_{f_n}}$. This proves (*).

\smallskip

If (a) and (b) are true, then (c) follows immediately from (*) applied with $X_i:=N_\L(P_i)$ for $i=0,1,\dots,n$. So it remains to prove (a) and (b). 

\smallskip

Let $P\in\Delta$. For any $g\in N_\L(P)$, we have $g^{-1}\in N_\L(P)$ as $c_{g^{-1}}=(c_g)^{-1}$. Moreover, if $w=(f_1,\dots,f_n)\in\W(N_\L(P))$, then $w\in\D=\D_\Delta$ via $P,P,\dots,P$ and, by (*) applied with $P_i:=X_i:=P$ for $i=0,\dots,n$, we have $\Pi(w)\in N_\L(P)$. Hence, $N_\L(P)$ is a subgroup of $\L$ and we have proved (a).

\smallskip

Let now $f\in\L$ with $P\subseteq\D(f)$ and $P^f\in\Delta$. Then for $x\in N_\L(P)$, we have $(f^{-1},x,f)\in\D$ via $P^f,P,P,P^f$. So $x\in\D(f)$ and, by (*), $x^f\in N_\L(P^f)$. Thus $c_f$ induces a map $N_\L(P)\rightarrow N_\L(P^f)$. Applying this property with $P^f$ and $f^{-1}$ in the roles of $P$ and $f$, we get that $c_{f^{-1}}$ induces a map $N_\L(P^f)\rightarrow N_\L(P)$. Since $c_f\colon \D(f)\rightarrow \D(f^{-1})$ is bijective with inverse map $c_{f^{-1}}$, it follows that $c_f$ induces a bijection $N_\L(P)\rightarrow N_\L(P^f)$. If $x,y\in N_\L(P)$, then $u=(f^{-1},x,f,f^{-1},y,f)\in\D$ via $P^f,P,P,P^f,P,P,P^f$. So using axiom (PG3) and Lemma~\ref{L:PartialGroup}(h), we conclude $(xy)^{c_f}=\Pi(f^{-1},x,y,f)=\Pi(u)=\Pi(x^f,y^f)=(x^{c_f})(y^{c_f})$. Hence, $c_f$ is a homomorphism of groups and we have proved (c). 
\end{proof}

\begin{lemma}\label{L:ConstructLocality1}
Let $\L$ be a partial group and let $S$ be a maximal $p$-subgroup of $\L$. Suppose $\Delta_0$ is a set of subgroups of $S$ such that $\D(\L)=\D(\L)_{\Delta_0}$ and $\Delta_0$ is closed under taking $\L$-conjugates in $S$. 
\begin{itemize}
\item [(a)] Let $\Delta$ be a set of subgroups of $S$ such that $\Delta_0\subseteq\Delta$ and, for every $P\in\Delta$, there exists $Q\in\Delta_0$ with $Q\leq P$. Then $\D(\L)=\D(\L)_{\Delta}$.
\item [(b)] Set $\Delta:=\{P\leq S\colon \exists Q\in\Delta_0\mbox{ with }Q\unlhd P\}$. Then $\Delta$ is closed under taking $\L$-conjugates in $S$.
\end{itemize}
\end{lemma}

\begin{proof}
(a) Let $\Delta$ be as in (a). As $\Delta_0\subseteq\Delta$, we have clearly $\D(\L)=\D(\L)_{\Delta_0}\subseteq\D(\L)_{\Delta}$. Let now $w=(f_1,\dots,f_n)\in\D(\L)_{\Delta}$ via some elements $P_0,P_1,\dots,P_n\in\Delta$. By assumption, there exists $Q_0\in\Delta_0$ with $Q_0\leq P_0$. For $i=1,\dots,n$ define $Q_i$ recursively by $Q_i:=Q_{i-1}^{f_i}$. Note that $Q_i\subseteq P_{i-1}^{f_i}=P_i\leq S$ and thus $Q_{i-1}\subseteq S_{f_i}$ for $i=1,\dots,n$. As $\Delta_0$ is closed under taking $\L$-conjugates in $S$ and $Q_0\in\Delta_0$, it follows that by induction that $Q_0,Q_1,\dots,Q_n\in\Delta_0$. Hence, $w\in\D(\L)_{\Delta_0}=\D(\L)$ via $Q_0,Q_1,\dots,Q_n$. This proves (a).

\smallskip

(b) Let now $\Delta$ be as in (b). Pick $P\in\Delta$ and $f\in\L$ with $P\subseteq S_f$. By definition of $\Delta$, there exists $Q\in\Delta_0$ such that $Q\unlhd P$. We apply now Lemma~\ref{L:ConstructLocality0} with $\Delta_0$ in place of $\Delta$. By part (a) of that lemma, $N_\L(Q)$ is a subgroup of $\L$. As $\Delta_0$ is closed under taking $\L$-conjugates in $S$ and $Q\leq P\subseteq S_f$, we have $Q^f\in\Delta_0$. So by Lemma~\ref{L:ConstructLocality0}(b), $c_f\colon N_\L(Q)\rightarrow N_\L(Q^f)$ is an isomorphism of groups. In particular, as $P$ is a subgroup of $N_\L(Q)$, $P^f$ is a subgroup of $N_\L(Q^f)$ and thus also of $S$. Moreover, $Q^f\unlhd P^f$ and thus $P^f\in\Delta$. This shows that $\Delta$ is closed under taking $\L$-conjugates in $S$.
\end{proof}

\begin{lemma}\label{L:ConstructLocality}
 Let $\L$ be a partial group and let $S$ be a maximal $p$-subgroup of $\L$. Suppose $\Delta_0$ is a set of subgroups of $S$ which is closed under taking $\L$-conjugates. Assume $\D(\L)=\D(\L)_{\Delta_0}$. Set
\[\Delta:=\{P\leq S\colon\exists Q\in\Delta_0\mbox{ with }Q\leq P\}.\]
Then $(\L,\Delta,S)$ is a locality.
\end{lemma}

\begin{proof}
Note that $\Delta$ is by construction closed under taking overgroups in $S$. Since $S$ is by assumption a maximal $p$-subgroup of $\L$, it remains thus to show that $\Delta$ is closed under taking $\L$-conjugates, and that $\D(\L)=\D(\L)_{\Delta}$. Since $\Delta_0$ is given, we can defined sets $\Delta_i$ for $i\geq 1$ recursively by 
\[\Delta_i:=\{P\leq S\colon \exists Q\in\Delta_{i-1}\mbox{ with }Q\unlhd P\}.\]
If $Q\in\Delta_0$ and $Q\leq P\leq S$, then $Q$ is subnormal in $P$, and the subnormal length is bounded by $|S|$. Hence, there exists $n\in\N$ with $\Delta_n=\Delta$. Therefore, it is sufficient to prove that, for all $i\geq 0$, $\D(\L)=\D(\L)_{\Delta_i}$ and $\Delta_i$ is closed under taking $\L$-conjugates in $S$. Using induction on $i$, this follows however immediately from Lemma~\ref{L:ConstructLocality1} and the fact that the claim is by assumption true for $i=0$.
\end{proof}

In Section~\ref{S:SemidirectProdLoc}, we will consider substructures of localities which are localities again, as introduced in the following definition.

\begin{definition}
We say that $(\H,\Delta_\H,S_\H)$ is a \textit{sublocality} of $(\L,\Delta,S)$ if $\H$ is a partial subgroup of $\L$, $S_\H=S\cap \H$, $\Delta_\H$ is a set of subgroups of $S_\H$ and, regarding $\H$ as a partial group with product $\Pi|_{\W(\H)\cap\D(\L)}$, the triple $(\H,\Delta_\H,S_\H)$ forms a locality. 
\end{definition}

We stress that, in the above definition, $\Delta_\H$ is not assumed to be a subset of $\Delta$. Such a condition would be too restrictive for our purposes as will become clear in Section~\ref{S:SemidirectProdLoc}.

\subsection{Homomorphisms of partial groups.}

\begin{definition}\label{D:PartialHom}
Let $\L$ and $\L'$ be partial groups, let $\varphi \colon \L \rightarrow \L',g\mapsto g^\phi$ be a mapping. By abuse of notation, let $\varphi$ also denote the induced map on words 
\[\W(\L) \rightarrow \W(\L'),\quad w=(f_1,\dots,f_n)\mapsto w^{\phi}=(f_1^\phi,\dots,f_n^\phi).\]
Accordingly, set $\D(\L)^\phi=\{w^\phi\colon w\in\D(\L)$. We say that $\varphi$ is a homomorphism of partial groups if 
\begin{enumerate}
\item $\D(\L)^{\varphi} \subseteq \D(\L')$; and 
\item $\Pi(w)^\varphi = \Pi'(w^\phi)$ for every $w \in \D(\L)$.
\end{enumerate}
If moreover $\varphi$ is bijective and $\D(\L)^\phi = \D(\L')$, then we say that $\varphi$ is an isomorphism of partial groups. The isomorphisms of partial groups from $\L$ to itself are called \emph{automorphisms} and the set of these automorphisms is denoted by $\Aut(\L)$. 
\end{definition}

\begin{notation}
Whenever $\phi\colon\L\rightarrow\L'$ is a homomorphism of partial groups, then (as in Definition~\ref{D:PartialHom}) by abuse of notation, we will denote by $\phi$ also the induced map on words
\[\W(\L) \rightarrow \W(\L'),\quad w=(f_1,\dots,f_n)\mapsto w^{\phi}=(f_1^\phi,\dots,f_n^\phi).\]
\end{notation}

\begin{lemma}\label{L:PartialHom}
If  $\varphi \colon \L \rightarrow \L'$ is a homomorphism of partial groups, then the following hold:
\begin{itemize}
 \item [(a)] $(g^{-1})^\varphi = (g^\varphi)^{-1}$ for every $g \in \L$.
 \item [(b)] If $u,v\in\W(\L)$ with $u\circ v\in\D(\L)$, then $\Pi(u\circ v)^\phi=\Pi'(\Pi(u)^\phi,\Pi(v)^\phi)$ where $\Pi$ and $\Pi'$ denote the partial products on $\L$ and $\L'$ respectively.
 \item [(c)] We have $\One^\phi=\One'$.
\end{itemize}
\end{lemma} 

\begin{proof}
For the proof of (a) see \cite[Lemma 1.13]{loc1}. For (b) note that $\Pi(u\circ v)^\phi=\Pi((u\circ v)^{\phi})=\Pi((u^{\phi})\circ (v^{\phi}))=\Pi(\Pi(u^{\phi}),\Pi(v^{\phi}))=\Pi(\Pi(u)^\phi,\Pi(v)^\phi)$, where the third equality uses Lemma~\ref{L:PartialGroup}(a). This proves (b). Property (c) holds as  $\One^\phi=\Pi(\emptyset)^\phi=\Pi'(\emptyset^\phi)=\Pi'(\emptyset)=\One'$. 
\end{proof}

\begin{lemma}\label{L:PartialIso}
 Let $\L$ and $\L'$ be partial groups and let $\varphi\colon \L\rightarrow\L'$ be a map. Then $\phi$ is an isomorphism of partial groups if and only if $\phi$ is bijective and $\phi$ and $\phi^{-1}$ are both homomorphisms of partial groups.
\end{lemma}

\begin{proof}
If $\phi$ is bijective and $\phi$ and $\phi^{-1}$ are both homomorphisms of partial groups, then $\D(\L)^\phi\subseteq\D(\L')$ and $\D(\L')^{\phi^{-1}}\subseteq \D(\L)$, with the latter inclusion implying $\D(\L')\subseteq\D(\L)^\phi$. Thus, we get $\D(\L)^\phi=\D(\L')$ and thus $\phi$ is a homomorphism of partial groups.

\smallskip

Assume now that $\phi$ is an isomorphism of partial groups. Then $\D(\L)^\phi=\D(\L')$ and thus $\D(\L')^{\phi^{-1}}=\D(\L)$. Given $w\in\D(\L')$, it remains to show that $\Pi'(w)^{\phi^{-1}}=\Pi(w^{\phi^{-1}})$. Note that $w^{\phi^{-1}}\in\D(\L)$ and thus, as $\phi$ is a homomorphism of partial groups, $\Pi(w^{\phi^{-1}})^\phi=\Pi'((w^{\phi^{-1}})^\phi)=\Pi'(w)$. This implies the required equation.
\end{proof}


\section{External semidirect products of partial groups}\label{S:ExternalPartialGroup}

In this section we will introduce the external semidirect product of partial groups as a natural generalization of the external semidirect product of groups. For that we need to consider the action of a partial group on another partial group as introduced in the following definition.

\begin{definition}\label{D:PartialGroupAction}
 Let $\X$ and $\N$ be partial groups. Then we say that \emph{$\X$ acts on the partial group $\N$} if there exists a homomorphism $\varphi \colon \X \rightarrow \Aut(\N)$ of partial groups. If such $\phi$ is given, we say also that $\X$ acts on the partial group $\N$ via $\phi$. 
\end{definition}

Note here that $\Aut(\N)$ forms a group with the composition of maps as multiplication; we regard $\Aut(\N)$ as a partial group in the usual way by extending the ``binary'' product on $\Aut(\N)$ to a multivariable product $\Pi_{\Aut(\N)}\colon\W(\Aut(\N))\rightarrow \Aut(\N)$.

\begin{remark}
If $X$ is a group, then $X$ acts on the partial group $\N$ via $\phi$ if and only if $\phi\colon X\rightarrow \Aut(\N)$ is a homomorphism of groups (cf. \cite[Lemma~1.16]{loc1}). 
\end{remark}

\begin{notation} 
Assume a partial group $\X$ acts on a partial group $\N$ via a homomorphism $\phi$. For $\Y\subseteq\X$ and $\M\subseteq\N$, set
\[\C_\Y^\phi(\M):=\{y\in\Y\colon m^{y^\phi}=m\mbox{ for all }m\in\M\}.\] 
If it does not lead to confusion, we write also $\C_\Y(\M)$ instead of $\C_\Y^\phi(\M)$.
\end{notation}

\begin{lemma}\label{L:ExternalCXNNormal}
If a partial group $\X$ acts on a partial group $\N$ via a homomorphism $\phi$, then $\C_\X^\phi(\N)=\ker(\phi)$ is a partial normal subgroup of $\X$.
\end{lemma}

\begin{proof}
Clearly, $\C_\X(\N)=\ker(\phi)$. By \cite[Lemma~1.14]{loc1}, the kernel of a homomorphism of partial groups is always a partial normal subgroup.
\end{proof}

We will now construct external semidirect products of partial groups. For that we will work under the following hypothesis.

\begin{hypo}\label{H:ExternalPartialGroups}
Let $\X$ and $\N$ be partial groups with products $\Pi_\X\colon\D(\X)\rightarrow \X$ and $\Pi_\N\colon\D(\N)\rightarrow \N$ respectively. Assume that $\X$ acts on the partial group $\N$ via $\varphi \colon \X \rightarrow \Aut(\N)$. So $x^\phi$ is an automorphism of $\N$ for every $x \in \X$. By abuse of notation, denote by $x^\varphi$ also the corresponding map induced on the set of words in $\N$:
\[ x^\varphi \colon \W(\N) \rightarrow \W(\N), \quad w = (f_1, \dots, f_n) \mapsto w^{x^\varphi}=(f_1^{x^\varphi}, \dots, f_n^{x^\varphi}).\]
\end{hypo}

\begin{lemma}\label{L:ActionPartialGroup}
Assume Hypothesis~\ref{H:ExternalPartialGroups}. Then the following hold:
\begin{itemize}
\item[(a)] For every $w \in \D(\N)$, we have $w^{x^\varphi} \in \D(\N)$ and $(\Pi_{\N}(w))^{x^\varphi} = \Pi_{\N}( w^{x^\varphi})$.
\item[(b)] For every $f \in \N$, we have  $(f^{-1})^{x^\varphi} = (f^{x^\varphi})^{-1}$.
\item[(c)] If $u,v\in\W(\X)$ with $u\circ v\in\D(\X)$, then $f^{\Pi(u\circ v)^\phi}=(f^{\Pi(u)^\phi})^{\Pi(v)^\phi}$.
\end{itemize}
\end{lemma}

\begin{proof}
As $x^\varphi$ is an automorphism of $\N$, property (a) follows from the definition of a homomorphism of partial groups, whereas (b) follows from Lemma~\ref{L:PartialHom}(a). As $\phi\colon \X\rightarrow \Aut(\N)$ is a homomorphism of partial groups, Lemma~\ref{L:PartialHom}(b) gives that $(\Pi(u\circ v))^\phi=\Pi_{\Aut(\N)}(\Pi(u)^\phi,\Pi(v)^\phi)$ is the composition of $\Pi(u)^\phi$ with $\Pi(v)^\phi$. So property (c) holds.
\end{proof}

\begin{definition}\label{D:ExternalNotationwxwn}
Assume Hypothesis~\ref{H:ExternalPartialGroups}.  
If $w = ( (x_1, f_1), \dots, (x_n, f_n))$ with $x_i \in \X$ and $f_i \in \N$, then we write
\[ w_\X = (x_1, \dots, x_n) \in \W(\X).\]
If $w_\X\in\D(\X)$, then set  
\[w_\N=( f_1^{(x_2 \cdots x_n)^\varphi}, f_2^{(x_3 \cdots x_n)^\varphi}, \dots, f_{n-1}^{x_n^\varphi}, f_n) \in \W(\N).\]
If $w$ is the empty word, we mean here that $w_\X$ and $w_\N$ are also both equal to the empty word.
\end{definition}

In the definition above note that, if $w_\X\in\W(\X)$, then for each $i=1,\dots,n$ the product $\Pi_\X(x_i,x_{i+1},\dots,x_n)=x_ix_{i+1}\cdots x_n$ is defined by axiom (PG1) of a partial group. Thus, the word $w_\N$ is in this case well-defined.

\begin{definition}[External semidirect product of partial groups]\label{external.par}
Assume Hypothesis~\ref{H:ExternalPartialGroups}. The external semidirect product of $\X$ with $\N$ (via $\varphi$) is the triple $(\L, \Pi, (-)^{-1}) $ where
\begin{itemize}
\item $\L= \{ (x, f) \mid x \in \X, f \in \N\}$;
\item$\D(\L) = \{ w \in \W(\L) \mid w_\X\in\D(\X)\mbox{ and }w_\N \in \D(\N)\}$;
\item $\Pi \colon \D(\L) \rightarrow \L, \quad w \mapsto (\Pi_\X(w_\X), \Pi_{\N}(w_\N))$; and 
\item $(-)^{-1} \colon \L \rightarrow \L, \quad (x,f) \mapsto (x,f)^{-1} = (x^{-1}, ((f^{-1})^{(x^{-1})^\varphi})$.
\end{itemize}
We write also $\X \ltimes_\varphi \N$ instead of $\L$.
\end{definition}

The next goal will be to show that the external semidirect product forms a partial group. We will need the following lemma.

\begin{lemma}\label{properties}
Assume Hypothesis~\ref{H:ExternalPartialGroups}. Let $u,v\in\W(\L)$. Then the following hold:
\begin{itemize}
\item[(a)] $(u\circ v)_\X=u_\X\circ v_\X$. If $(u\circ v)_\X\in\D(\X)$, then $u_\X,v_\X\in\D(\X)$ and $(u \circ v)_\N = u_\N^{(\Pi_\X(v_\X))^\varphi} \circ v_\N$;
\item[(b)] $(u^{-1})_\X=(u_\X)^{-1}$ and $(u^{-1}\circ u)_\X=(u_\X)^{-1}\circ u_\X$. In particular, if $(u^{-1}\circ u)_\X\in\D(\X)$, then $u_\X\in\D(\X)$ and $(u^{-1} \circ u)_\N = (u_\N)^{-1} \circ u_\N$.
\end{itemize}
\end{lemma}

\begin{proof}
Write $u= ( (x_1, f_1), \dots, (x_n, f_n))$ and $v = ( (y_1, g_1), \dots, (y_m, g_m))$ with $x_i, y_i \in \X$ and $f_i,g_i \in \N$. Clearly $(u\circ v)_\X=u_\X\circ v_\X$. Assume now that $(u\circ v)_\X\in\D(\X)$. By axiom (PG1) of a partial group, we have  $u_\X,v_\X\in\D(\X)$. So $(u\circ v)_\N$, $u_\N$ and $v_\N$ are well-defined. To show the last part of (a), set $y = \Pi_\X(v_\X) =y_1y_2\cdots y_m$. Using Lemma~\ref{L:ActionPartialGroup}(c) for the second equality, we see that 
\begin{align*} 
(u \circ v)_\N &= ( f_1^{(x_2 \cdots x_n \cdot y_1\cdots y_m)^\varphi}, f_2^{(x_3 \cdots x_n \cdot y_1\cdots y_m)^\varphi}, \dots, f_{n-1}^{(x_n \cdot y_1\cdots y_m)^\varphi}, f_n^{(y_1\cdots y_m)^\varphi}) \circ v_\N \\
&= ( (f_1^{(x_2 \cdots x_n )^\varphi})^{y^\varphi}, (f_2^{(x_3 \cdots x_n)^\varphi})^{y^\varphi}, \dots, (f_{n-1}^{x_n^\varphi})^{y^\varphi}, f_n^{y^\varphi}) \circ v_\N \\
&= u_\N^{y^\varphi} \circ v_\N. 
\end{align*}
Notice that 
\[u^{-1}=((x_n^{-1},(f_n^{-1})^{(x_n^{-1})^\phi}),\dots,(x_1^{-1},(f_1^{-1})^{(x_1^{-1})^\phi})).\] 
In particular, $(u^{-1})_\X=(x_n^{-1},\dots,x_1^{-1})=(u_\X)^{-1}$ and $(u^{-1}\circ u)_\X=(u^{-1})_\X\circ u_\X=(u_\X)^{-1}\circ u_\X$. Assume now $(u^{-1}\circ u)_\X\in\D(\X)$. Then by axiom (PG1) of a partial group, $u_\X,(u^{-1})_\X\in\D(\X)$ and thus $u_\N,(u^{-1})_\N$ are well-defined. Using Lemma~\ref{L:ActionPartialGroup}(c), it follows 
\[(u^{-1})_\N = ((f_n^{-1})^{(x_n^{-1}x_{n-1}^{-1}\cdots x_1^{-1})^\phi},(f_{n-1}^{-1})^{(x_{n-1}^{-1}x_{n-2}^{-1}\cdots x_1^{-1})^\phi},\dots,(f_1^{-1})^{(x_1^{-1})^\phi}).\]
Using Lemma~\ref{L:ActionPartialGroup}(b),(c), we conclude 
\[(u^{-1})_\N^{\Pi_\X(u_\X)}=(u^{-1})_\X^{(x_1x_2\cdots x_n)^\phi}=(f_n^{-1},(f_{n-1}^{-1})^{x_n^\phi},\dots,(f_1^{-1})^{(x_2\cdots x_n)^\phi})=(u_\N)^{-1}.\] 
So by part (a), $(u^{-1}\circ u)_\N=(u^{-1})_\N^{\Pi_\X(u_\X)}\circ u_\N=(u_\N)^{-1}\circ u_\N$. This completes the proof.
\end{proof}

\begin{lemma}\label{partialgrp}
If a partial group $\X$ acts on a partial group $\N$ via $\phi$, then the external semidirect product $\X \ltimes_\varphi \N$ of $\X$ with $\N$ is a partial group.
\end{lemma}

\begin{proof}
Adopt the notation introduced in Hypothesis~\ref{H:ExternalPartialGroups}. We prove that the triple $(\L, \Pi, (-)^{-1}) $ defined in Definition~\ref{external.par} satisfies all the axioms of Definition \ref{partial}.
As $\N$ and $\X$ are non-empty, $\L$ is non-empty. As usual, we regard elements of $\L$, $\X$ or $\N$ as words of length one.
Then, for any $(x,f) \in \L$, we have $(x,f)_\X=x\in\D(\X)$ and $(x,f)_\N =f\in \D(\N)$, which implies $(x,f)\in\D(\L)$. Thus $\L \subseteq \D(\L)$. 
Also, $\Pi( (x,f)) = (\Pi_\X(x), \Pi_{\N}(f)) = (x,f)$. So $\Pi$ restricts to the identity map on $\L$. 

\smallskip

Let now $u,v\in\W(\L)$ such that $u \circ v \in \D(\L)$. By the first part of Lemma~\ref{properties}(a) we get $u_\X,v_\X\in\D(\X)$. Set $y= \Pi_\X(v_\X)$. Using the second part of Lemma~\ref{properties}(a), it follows then $ (u \circ v)_\N =  u_\N^{y^\varphi} \circ v_\N  \in \D(\N)$.
Since $\N$ is a partial group we deduce that $u_\N^{y^\varphi}, v_\N \in \D(\N)$ and since $y^\varphi$ is an automorphism of $\N$ we also get $u_\N \in \D(\N)$. Hence $u,v \in \D(\L)$ by the definition of $\D(\L)$. So we have shown that axioms (PG1) and (PG2) of Definition~\ref{partial} hold.

\smallskip

Let now $u,v,w\in\W(\L)$ such that $u \circ v \circ w \in \D(\L)$. By what we have just shown, we have $u,v,w\in\D(\L)$. In particular, $v_\X,w_\X\in\D(\X)$. Set $y= \Pi_\X(v_\X)$ and $z = \Pi_\X(w_\X)$. By Lemma \ref{properties}(a) applied twice we get $(u\circ v\circ w)_\X=u_\X\circ v_\X\circ w_\X$ and  $(u \circ v \circ w)_\N =  u_\N^{(yz)^\varphi} \circ v_\N^{z^\varphi} \circ w_\N$ (and all the terms in the latter equation are well-defined). So by definition of $\D(\L)$, we have $u_\X\circ v_\X\circ w_\X\in\D(\X)$ and $u_\N^{(yz)^\varphi} \circ v_\N^{z^\varphi} \circ w_\N \in \D(\N)$. Since $\X$ is a partial group, it follows 
\[u_\X\circ(y)\circ w_\X\in\D(\X)\mbox{ and }\Pi_\X(u_\X\circ v_\X\circ w_\X)=\Pi(u_\X\circ(y)\circ w_\X).\]
Similarly, since $\N$ is a partial group, we can conclude that  
\[u_\N^{(yz)^\varphi} \circ (\Pi_{\N}(v_\N^{z^\varphi})) \circ w_\N \in \D(\N)\]
and 
\[\Pi(u_\N^{(yz)^\varphi} \circ (\Pi_{\N}(v_\N^{z^\varphi})) \circ w_\N)=\Pi(u_\N^{(yz)^\varphi} \circ v_\N^{z^\varphi} \circ w_\N).\] 
Recall that $v \in \D(\L)$. By definition of the product on $\L$, we have $\Pi(v) = (y,\Pi_{\N}(v_\N))$. Using this and Lemma~\ref{properties}(a) twice we observe  
\[  (u \circ (\Pi(v)) \circ w)_\N = u_\N^{(yz)^\varphi} \circ ( \Pi_{\N}(v_\N)^{z^\varphi}) \circ w_\N = u_\N^{(yz)^\varphi} \circ (\Pi_{\N}(v_\N^{z^\varphi})) \circ w_\N \in \D(\N), \]
where the last equality uses Lemma~\ref{L:ActionPartialGroup}(a) (i.e. the fact that $z^\phi$ is an automorphism of $\N$). Recall also that $(u\circ(\Pi(v))\circ w)_\X=u_\X\circ (y)\circ w_\X\in\D(\X)$. 
Hence $u \circ \Pi(v) \circ w \in \D(\L)$ by definition of $\D(\L)$. Also,
\begin{eqnarray*}
\Pi( u \circ v \circ w) &=& (\Pi_\X(u_\X\circ v_\X\circ w_\X), \Pi_{\N}(  u_\N^{(yz)^\varphi} \circ v_\N^{z^\varphi} \circ w_\N) ) \\
&=& (\Pi_\X(u_\X\circ (y)\circ w_\X), \Pi_{\N}( u_\N^{(yz)^\varphi} \circ (\Pi_{\N}(v_\N^{z^\varphi})) \circ w_\N ) ) \\
&=& (\Pi_\X((u\circ (\Pi(v))\circ w)_\X),\Pi_\N((u\circ (\Pi(v))\circ w)_\N))\\
&=& \Pi( u \circ \Pi(v) \circ w).
\end{eqnarray*}
This shows axiom (PG3) in Definition~\ref{partial}. 

\smallskip

To show the final axiom, suppose now $u \in \D(\L)$, i.e. $u_\N \in \D(\N)$. By Lemma \ref{properties}(b) and the assumption that $\X$ and $\N$ are partial groups, we get $(u^{-1}\circ u)_\X=u_\X^{-1}\circ u_\X\in\D(\X)$ and $(u^{-1} \circ u )_\N =  u_\N^{-1} \circ u_\N \in \D(\N)$. Moreover, $\Pi_\X(u_\X^{-1}\circ u_\X)=\One_\X:=\Pi_\X(\emptyset)$ and $\Pi_\N(u_\N^{-1}\circ u_\N)=\One_\N:=\Pi_\N(\emptyset)$. Hence $u^{-1} \circ u  \in \D(\L)$ and
\[ \Pi(u^{-1} \circ u ) = (\Pi_\X(u_\X^{-1}\circ u_\X), \Pi_{\N}(  u_\N^{-1} \circ u_\N) ) = (\One, \One)=(\Pi_\X(\emptyset),\Pi_\N(\emptyset))=\Pi(\emptyset)=:\One_\L.\]
Therefore the set $\L$ with the product $\Pi$ and the inversion $(-)^{-1}$ is a partial group.
\end{proof}

\begin{remark}
\begin{itemize}
\item If $\X$ and $\N$ are groups and we regard $\X$ and $\N$ as partial groups in the natural way, then the group automorphisms of $\N$ are precisely the automorphisms of the partial group $\N$. Moreover, a map $\phi\colon \X\rightarrow \Aut(\N)$ is a homomorphism of partial groups if and only if $\phi$ is a homomorphism of groups. If so, then $\X \ltimes_\varphi \N$ is the usual external semidirect product of groups.
\item If $\X$ and $\N$ are partial groups and $\phi\colon\X\rightarrow \Aut(\N)$ maps every element of $\X$ to the identity, then $\phi$ is a homomorphism of partial groups and the external semidirect product $\X\ltimes_\varphi\N$ is the same as the external direct product $\X\times \N$ of partial groups as introduced in \cite{direct.prod}.
\end{itemize}
\end{remark}

Similarly as in the case of external semidirect products of groups, $\X$ and $\N$ can be identified with partial subgroups of the external semidirect product $\X \ltimes_\varphi \N$. More generally this holds for partial subgroups $\Y$ and $\M$ of $\X$ and $\N$ respectively. We will use the following notation.

\begin{notation}\label{N:ExternalCartesianProd}
Assume Hypothesis~\ref{H:ExternalPartialGroups} and suppose $\L =\X\ltimes_\varphi \N$. For every partial subgroup $\Y$ of $\X$ and every partial subgroup $\M$ of $\N$ we set
\[(\Y,\M) = \{ (y,m) \in \L \mid y \in \Y,\;m\in\M \}.\]
\end{notation}
Note that $(\Y,\M)$ is actually the same as the Cartesian product of $\Y$ and $\M$, which is usually denoted by $\Y\times\M$. However, we wish to avoid this notation as it would lead to confusion with our notation of the direct product $\Y\times\M$ of partial groups. 

Following the usual notation for (binary) groups, we write $\One$ for the partial subgroup $\{\One\}$ of any partial group with identity $\One$; note that this is a partial subgroup by Lemma~\ref{L:PartialGroup}(e). In particular, $(\Y,\One)$ and $(\One,\M)$ are defined.

\begin{lemma}\label{subgroups}
Assume Hypothesis~\ref{H:ExternalPartialGroups} and suppose $\L = \X \ltimes_\varphi \N$. Let $\Y$ be a partial subgroup of $\X$ and let $\M$ be a partial subgroup of $\N$. Then the following hold:
\begin{itemize}
\item[(a)] If $\Y\subseteq \rN_\X(\M)$, then $(\Y,\M)$ is a partial subgroup of $\L$; 
\item[(b)] $(\Y,\One)$ is a partial subgroup of $\L$;
\item[(c)] $(\One,\M)$ is a partial subgroup of $\L$.
\item[(d)] The maps $\alpha\colon \Y\rightarrow (\Y,\One),y\mapsto (y,\One)$ and $\M\rightarrow (\One,\M),m\mapsto (\One,m)$ are isomorphisms of partial groups.
\end{itemize}
\end{lemma}

\begin{proof}
(a) Suppose $\Y\subseteq \C_\X(\M)$. Using the definition of $\D(\L)$, we get that, for any word $w=((y_1,m_1),\dots,(y_k,m_k))\in\D(\L)\cap\W((\Y,\M))$ with $y_1,\dots,y_k\in\Y$ and $m_1,\dots,m_k\in\M$, we have $w_\X=(y_1,\dots,y_m)\in\D(\X)\cap\W(\Y)$ and $w_\N=(m_1^{(y_2 \cdots y_k)^\varphi}, m_2^{(y_3 \cdots y_k)^\varphi}, \dots, m_{k-1}^{y_k^\varphi},m_k)\in\D(\N)\cap \W(\M)$, as $\Y$ is a partial subgroup of $\X$ with $\Y\subseteq\rN_\X(\M)$. So by definition of $\Pi$, we have $\Pi(w)=(\Pi_\X(w_\X),\Pi_\N(w_\N))\in(\Y,\M)$ as $\Y$ and $\M$ are partial subgroups. Moreover, if $(y,m)\in(\Y,\M)$ with $y\in\Y$ and $m\in\M$, then $y^{-1}\in\Y$ and $m^{-1}\in\M$. Thus $(m^{-1})^{(y^{-1})^\phi}=m^{-1}$ and  $(y,m)^{-1}=(y^{-1},m^{-1})\in(\Y,\M)$. This proves (a). Properties (b) and (c) follow from (a).

\smallskip

(d) Note that $\alpha$ and $\beta$ are clearly bijective. Let $u=(y_1,\dots,y_k)\in\W(\Y)$. Then $u^\alpha=((y_1,\One),\dots,(y_k,\One))$, so $(u^\alpha)_\X=(y_1,\dots,y_k)=u$. Moreover, if $(u^\alpha)_\X=u\in\D(\X)$ so that $(u^\alpha)_\N$ is well-defined, then $(u^\alpha)_\N=(\One,\dots,\One)$ by Lemma~\ref{L:PartialHom}(c). Hence, by Lemma~\ref{L:PartialGroup}(e), we have then $(u^\alpha)_\N\in\D(\N)$ and $\Pi_\N((u^\alpha)_\N)=\One$. So by (SD3), $u\in\D(\X)$ if and only if $u^\alpha\in\D(\L)$. Hence, $u\in\D(\Y)=\D(\X)\cap\W(\Y)$ if and only if $u^\alpha\in\D((\Y,\One))=\D(\L)\cap\W((\Y,\One))$. Since $\alpha$ is bijective, this shows $\D(\Y)^\alpha=\D((\Y,\One))$. Moreover, if $u\in\D(\Y)$, then $\Pi(u^\alpha)=(\Pi_\X((u^\alpha)_\X),\Pi_\N((u^\alpha)_\N))=(\Pi_\X(u),\One)=\Pi_\X(u)^\alpha$. This proves that $\alpha$ is an isomorphism of partial groups. Similar arguments show that $\beta$ is an isomorphism of partial groups.  
\end{proof}

\section{Internal semidirect products of partial groups}\label{S:InternalPartialGroup}

\begin{definition}[Internal semidirect products of partial groups]\label{internal.par}

Let $\L$ be a partial group, let $\X$ and $\N$ be a partial subgroups of $\L$. Assume
\begin{itemize}
\item[(SD1)] $\X \subseteq \mathrm{N}_\L(\N)$;
\item[(SD2)] for every $g \in \L$ there is a unique $x\in \X$ and a unique $n\in\N$ such that $(x,n)\in\D(\L)$ and $g= \Pi( x,n )$;
\end{itemize}
For every word $w=(\Pi(x_1,n_1), \dots, \Pi(x_k,n_k) )\in\W(\L)$ with $x_1,\dots,x_k\in \X$ and $n_1,\dots,n_k\in\N$ set $w_\X:=(x_1,\dots,x_k)$. Moreover, if $w_\X\in\D(\L)$, set 
\[w_\N:=(n_1^{\Pi(x_2, \dots, x_k)}, n_2^{ \Pi(x_3, \dots, x_k)}, \dots ,n_k ).\]

\smallskip

We say that $\L$ is the \emph{internal semidirect product} of $\X$ with $\N$ if in addition to (SD1) and (SD2) the following property holds: 
\begin{itemize}
\item[(SD3)] For every word $w\in\W(\L)$ we have $w\in\D(\L)$ if and only if $w_\X\in \D(\L)$ and $w_\N\in\D(\L)$; and in this case 
\[ \Pi(w) = \Pi(\Pi(w_\X), \Pi(w_\N)).\]
\end{itemize}
\end{definition}

\begin{remark}
\begin{itemize}
\item [(a)] If $\L$ is the internal semidirect product of $\X$ with $\N$, and $\X$ and $\N$ are subgroups, then $\L$ is by (SD3) a group. Indeed, $\L$ is the internal semidirect product of groups in the usual definition; to see this use Lemma~\ref{internal.prop.par}(b),(d) below.
\item [(b)] The internal direct product of partial groups as defined in \cite{direct.prod} is a special case of the internal semidirect product as defined above. For, if $\L$ is the internal direct product of $\X$ and $\N$, then by \cite[Lemma~6.3]{direct.prod}, $\X\subseteq \C_\L(\N)$ and so in particular (SD1) holds. Moreover, if $\X$ and $\N$ are partial subgroups of $\L$ with $\X\subseteq \C_\L(\N)$, then (SD2) and (SD3) are equivalent to saying that $\L$ is the internal direct product of $\X$ and $\N$ as defined in \cite[Definition~6.1]{direct.prod}.
\end{itemize} 
\end{remark}

\begin{lemma}\label{internal.prop.par}
Let $\L$ be  the internal semidirect product of $\X$ with $\N$. Then
\begin{itemize}
\item[(a)] $(x,n), (n, x) \in \D(\L)$ for every $x \in \X$ and every $n \in \N$; and
\item[(b)] $\L= \X\N$ and  $\X \cap \N = \One$.
\item[(c)] If $x\in \X$ and $n\in\N$, then 
\[\Pi(x,n)^{-1}=\Pi(x^{-1},(n^{-1})^{x^{-1}})\mbox{ and }\Pi(n,x)=\Pi(x,n^x).\]
\item[(d)] $\N$ is a partial normal subgroup of $\L$. 
\end{itemize}
\end{lemma}

\begin{proof}~
(a) Since $\X \subseteq \N_\L(\N)$, for every $n\in \N$ and every $y\in \X$ we have $(y^{-1}, n, y) \in \D(\L)$. Thus $(y^{-1}, n), (n,y) \in \D(\L)$. Every element $x\in \X$ can be written as $x = (x^{-1})^{-1}$, where $x^{-1}\in \X$ since $\X$ is a partial subgroup. Therefore, for every $x \in \X$ we get $(x,n),(n,x) \in \D(\L)$.

\smallskip

(b) By Axiom (SD2) in Definition \ref{internal.par}, we get $\L=\X \N$. Suppose $g \in \X \cap \N$. Then $g = \Pi(g,\One) = \Pi(\One,g)$. Since every element of $\L$ can be written in a unique way as a product of an element in $\X$ and an element in $\N$ we deduce that $g=\One$. Hence $\X \cap\N = \One$.

\smallskip

(c) Let $x\in\X$ and $n\in\N$. As $\X\subseteq N_\L(\N)$, we have $(x^{-1},n,x)\in\D(\L)$. So by Lemma~\ref{L:PartialGroup}(b), we have $w:=(x,x^{-1},n,x)\in\D(\L)$. So by axiom (PG3), $\Pi(n,x)=\Pi(w)=\Pi(x,n^x)$. Applying this property with $(n^{-1},x^{-1})$ in place of $(n,x)$ and using Lemma~\ref{L:PartialGroup}(g), we obtain $\Pi(x,n)^{-1}=\Pi(n^{-1},x^{-1})=\Pi(x^{-1},(n^{-1})^{x^{-1}})$. 

\smallskip

(d) Let $m\in\N$ and $g\in\L$ with $w:=(g^{-1},m,g)\in\D(\L)$. We need to show that $m^g\in\N$. By (SD2), there is $x\in\X$ and $n\in\N$ such that $g=\Pi(x,n)$. By (c), $g^{-1}=\Pi(x^{-1},(n^{-1})^{x^{-1}})$. Moreover, $m=\Pi(\One,m)$ with $\One\in\X$. As $w\in\D(\L)$, we have $w_\X,w_\N\in\D(\L)$ and $\Pi(w)=\Pi(\Pi(w_\X),\Pi(w_\X))$ by Axiom (SD3). By Lemma~\ref{L:PartialGroup}(f), we have $\Pi(w_\X)=\One$. Note also that $\Pi(w_\N)\in\N$, as $\N$ is a partial subgroup. Hence, $m^g=\Pi(w)=\Pi(\One,\Pi(w_\N))=\Pi(w_\N)\in\N$. This proves (d). 
\end{proof}

We show next that external semidirect products of partial groups provide natural examples of internal semidirect products.

\begin{theorem}\label{external.internal.par}
If $\X$ and $\N$ are partial groups and $\phi\colon \X\rightarrow \Aut(\N)$ is a homomorphism of partial groups, then the semidirect product $\L = \X \ltimes_\varphi \N$ is the internal semidirect product of $(\X,\One)$ with $(\One,\N)$. Moreover, $(\One,n)^{(x,\One)}=(\One,n^{x^\phi})$ and $(x,\One)(\One,n)=(x,n)$ for all $x\in \X$, $n\in\N$.
\end{theorem}

\begin{proof}
First notice that $(\X, \One)$ and $(\One,\N)$ are partial subgroups of $\L$ by Lemma \ref{subgroups}(b),(c).
Let $x\in \X$ and $n\in\N$. Set $s:=((x,\One)^{-1},(\One,n),(x,\One))$. Note that $(x,\One)^{-1} = (x^{-1},\One)$ and 
\[ s_\X=(x^{-1},\One,x)\mbox{ and }s_\N = (\One, n^{x^\varphi}, \One).\]
By axioms (PG1),(PG2) and Lemma~\ref{L:PartialGroup}(d), we have $s_\N\in\D(\N)$ with $\Pi_\N(s_\N)=n^{x^\phi}$. Moreover, by Lemma~\ref{L:PartialGroup}(f), we have $s_\X\in\D(\X)$ with $\Pi_\X(s_\X)=\One$. Thus, $s\in\D(\L)$ and $(\One,n)^{(x,\One)}=\Pi(s)=(\Pi_\X(s_\X),\Pi_\N(s_\N))=(\One,n^{x^\phi})$. In particular, $(\X,\One) \subseteq \mathrm{N}_\L((\One,\N))$. 

\smallskip

Using axiom (PG1) and Lemma~\ref{L:PartialGroup}(d) observe that, for every $x\in \X$ and every $n \in\N$, we have $t:=( (x,\One), (\One,n) ) \in \D(\L)$, since $t_\X:=(x,\One)\in\D(\X)$ and $t_\N:=(\One,n)\in\D(\N)$, and that moreover $\Pi( (x,\One), (\One,n))=\Pi(t)=(\Pi(t_\X),\Pi(t_\N))= (x,n)$. It follows that every element of $\L$ can be written in a unique way as a product of an element in $(\X,\One)$ and an element in $(\One,\N)$. 

\smallskip

Let $w = ( (x_1, n_1), \dots,  (x_k,n_k)) \in \W(\L)$ with $x_1,\dots,x_k\in\X$ and $n_1,\dots,n_k\in\N$. Using the notation introduced in  Definition~\ref{internal.par} set 
\[u:=w_{(\X,\One)}=((x_1,\One),\dots,(x_k,\One)).\] 
Now using the notation introduced in Definition~\ref{D:ExternalNotationwxwn}, we have $u_\X=(x_1,\dots,x_k)=w_\X$ and (if $u_\X\in\D(\X)$ and thus $u_\N$ is defined) $u_\N:=(\One,\dots,\One)\in\D(\N)$ by Lemma~\ref{L:PartialGroup}(e). So by definition of $\D(\L)$, we have $w_\X\in\D(\X)$ if and only if $u=w_{(\X,\One)}\in\D(\L)$. If so, then again using the notation introduced in Definition~\ref{internal.par},  
\[v:=w_{(\One,\N)}:=( (\One,n_1)^{\Pi( (x_2,\One), \dots, (x_k,\One) )},(\One,n_2)^{\Pi( (x_3,\One), \dots, (x_k,\One) )}, \dots, (\One,n_k) ) \in \W(\L).\]
By the property we proved above, we have $v=((\One,n_1^{(x_2\dots x_k)^\phi}),(\One,n_2^{(x_3\dots x_k)^\phi}),\dots,(\One,n_k))$. So $v_\X=(\One,\dots,\One)\in\D(\X)$ by Lemma~\ref{L:PartialGroup}(e), and 
\[w_\N = (n_1^{(x_2\dots x_k)^\phi},n_2^{(x_3\dots x_k)^\phi},\dots,n_k)=v_\N.\]
So, by definition of $\D(\L)$, we have $w_\N\in\D(\N)$ if and only if $v=w_{(\One,\N)}\in \D(\L)$. Altogether, as $w\in\D(\L)$ if and only $w_\X\in\D(\X)$ and $w_\N\in\D(\N)$, it follows that $w\in \D(\L)$ if and only if $u=w_{(\X,\One)}\in\D(\L)$ and $v=w_{(1,\N)} \in \D(\L)$. Moreover, if this is the case, we have 
\begin{eqnarray*}
\Pi(w) &=& (\Pi_\X(w_\X), \Pi_\N(w_\N)) \\
&=& \Pi((\Pi_\X(w_\X),\One),(\One,\Pi_\N(w_\N)))\\
&=& \Pi((\Pi_\X(u_\X),\Pi(u_\N)),(\Pi(v_\X),\Pi_\N(v_\N)))\\
&=& \Pi( \Pi(u),\Pi(v))\\
&=& \Pi(\Pi(w_{(\X,\One)}),\Pi(w_{(\One,\N)})),
\end{eqnarray*}
where the third equality uses Lemma~\ref{L:PartialGroup}(e). This proves that $\L$ is the internal semidirect product of $(\X,\One)$ with $(\One,\N)$.
\end{proof}

We will now show that internal semidirect products also lead to external semidirect products. More precisely, given a partial group $\L$ which is an internal semidirect product of a partial subgroup $\X$ with a partial subgroup $\N$, we show that $\X$ acts on $\N$ via conjugation.

\begin{lemma}\label{L:InternalExternalAction}
Let $\X$ and $\N$ be partial subgroups of a partial group $\L$ such that $\L$ is the internal semidirect product of $\X$ with $\N$. Then for every $x\in\X$ the map $c_x\colon \N\rightarrow \N,n\mapsto n^x$ is well-defined and an automorphism of the partial group $\N$. Moreover, $\phi\colon \X\rightarrow \Aut(\N),x\mapsto c_x$ is a homomorphism of partial groups. In particular, for all $x,y\in\X$ and $n\in\N$, we have $(n^x)^y=n^{xy}$. 
\end{lemma}

\begin{proof}
 Note that $\N$ is a partial group with product defined on $\D(\N):=\D(\L)\cap\W(\N)$. Let $u=(n_1,\dots,n_k)\in\D(\N)$, $x\in\X$ and set \[w:=(x^{-1},n_1,x,x^{-1},n_2,x,\dots,x^{-1},n_k,x).\]
Note that $c_x$ is well-defined by (SD1). By abuse of notation, we also write $c_x$ for the induced map on words in $\N$; in particular $u^{c_x}=(n_1^{c_x},\dots,n_k^{c_x})$. We show first that $c_x$ is a homomorphism of partial groups by proving that $u^{c_x}\in\D(\L)$ and $\Pi(u^{c_x})=\Pi(u)^{c_x}$. As a first step, we prove that $w\in\D(\L)$. Note that, $x=\Pi(x,\One)$ and $x^{-1}=\Pi(x^{-1},\One)$ where $\One\in\N$. Similarly, for every $n\in\N$, we have $n=\Pi(\One,n)$ with $\One\in\X$. So we conclude
\[w_\X=(x^{-1},\One,x,x^{-1},\One,x,\dots,x^{-1},\One,x)\mbox{ and }w_\N=(\One,n_1,\One,\One,n_2,\One,\dots,\One,n_k,\One).\]
By Lemma~\ref{L:PartialGroup}(f), we have $w_\X\in\D(\L)$, and by part (d) of the same lemma, $w_\N\in\D(\L)$. Hence, by (SD3), we have \[w\in\D(\L).\] 
Using axiom (PG3) several times, we conclude $u^{c_x}=(n_1^x,n_2^x,\dots,n_k^x)\in\D(\L)$ and $\Pi(u^{c_x})=\Pi(w)=\Pi(x^{-1},n_1,\dots,n_k,x)=\Pi(x^{-1},\Pi(u),x)=\Pi(u)^x=\Pi(u)^{c_x}$. This proves that $c_x$ is a homomorphism of partial groups from $\N$ to $\N$. As $x\in\X$ was arbitrary, it follows that $c_{x^{-1}}$ is also a homomorphism of partial groups from $\N$ to $\N$. By \cite[Lemma~1.6(c)]{loc1}, $c_x$ is bijective with $(c_x)^{-1}=c_{x^{-1}}$. This implies that $c_x$ is an automorphism of $\N$.

\smallskip

Thus, we know now that the map $\phi\colon\X\rightarrow\Aut(\N),x\mapsto c_x$ is well-defined. It remains to show that $\phi$ is a homomorphism of partial groups. Note that $\X$ forms a partial group with product defined on $\D(\X):=\D(\L)\cap \W(\X)$. Let $v=(x_1,\dots,x_l)\in\D(\X)$. As $\Aut(\N)$ is a group, $v^{\phi}:=(x_1^\phi,\dots,x_l^\phi)=(c_{x_1},\dots,c_{x_l})$ is in the domain $\W(\Aut(\N))$ of the naturally defined multivariable product $\Pi_{\Aut(\N)}$ on $\Aut(\N)$. So it remains only to show that $\Pi_{\Aut(\N)}(v^\phi)=\Pi(v)^\phi$, i.e. that $c_{x_1}c_{x_2}\cdots c_{x_l}=c_{\Pi(v)}$. Let $n\in\N$ and set 
\[w^*=(x_l^{-1},\dots,x_1^{-1},n,x_1,\dots,x_l)=v^{-1}\circ (n) \circ v.\]
Then $w^*_\X=(x_l^{-1},\dots,x_1^{-1},\One,x_1,\dots,x_l)=v^{-1}\circ (\One)\circ v$ and $w^*_\N=(\One,\dots,\One,n,\One,\dots,\One)$. 
Using Lemma~\ref{L:PartialGroup}(d) as well as axioms (PG4) and (PG1) of a partial group, we see that $w^*_\X\in\D(\L)$ and $w^*_\N\in\D(\L)$. Hence, $w^*\in\D(\L)$ by (SD3). By \cite[Lemma~1.4(f)]{loc1}, we have $\Pi(v^{-1})=\Pi(v)^{-1}$. Using axiom (PG3) several times, we conclude $n^{c_{x_1}c_{x_2}\cdots c_{x_l}}=\Pi(w^*)=\Pi(\Pi(v^{-1}),n,\Pi(v))=\Pi(\Pi(v)^{-1},n,\Pi(v))=n^{\Pi(v)}=n^{c_{\Pi(v)}}$. As $n\in\N$ was arbitrary, this shows $c_{x_1}c_{x_2}\cdots c_{x_l}=c_{\Pi(v)}$ as required. Thus, the proof is complete.
\end{proof}

\begin{cor}\label{C:CXNnormalX}
Suppose $\L$ is the internal semidirect product of a partial subgroup $\X$ with a partial subgroup $\N$. Then $\C_\X(\N)$ is a partial normal subgroup of $\X$.
\end{cor}

\begin{proof}
By Lemma~\ref{L:InternalExternalAction}, the map $\phi\colon\X\rightarrow \Aut(\N),x\mapsto c_x$ is a homomorphism of partial groups, i.e. $\X$ acts on $\N$ via $\phi$. So by Lemma~\ref{L:ExternalCXNNormal}, $\C_\X(\N):=\{x\in\X\colon n^x=n\mbox{ for all }n\in\N\}=\C_\X^\phi(\N)$ is a partial normal subgroup of $\X$.
\end{proof}

Lemma~\ref{L:InternalExternalAction} means that, given a partial group $\L$ which is an internal semidirect product of a partial subgroup $\X$ with a partial subgroup $\N$, it is indeed true that $\X$ acts on $\N$ in the sense of Definition~\ref{D:PartialGroupAction}. Thus, we are in a situation where we can also form the external semidirect product of $\X$ with $\N$. The next goal will be to show that such an external semidirect product will be isomorphic to $\L$. We will need the following lemma.  

\begin{lemma}\label{iso.par}
Let $\L$ and $\L'$ be partial groups which are internal semidirect products of $\X$ with $\N$ and of $\X'$ with $\N'$, respectively. Suppose that there exist isomorphisms of partial groups $\alpha \colon \X \rightarrow \X'$ and $\beta \colon \N \rightarrow \N'$ such that
\[  (n^\beta)^{x^\alpha} = (n^x)^\beta \text{ for every } x\in \X \text{ and every } n\in \N.\]
Then the mapping
\begin{align*} \varphi \colon \L \longrightarrow ~ \L'\mbox{ defined by }\Pi(x,n) \mapsto ~ \Pi'(x^\alpha, n^\beta) \mbox{ for all }x\in\X,\;n\in\N
\end{align*}
is an isomorphism of partial groups.
\end{lemma}

\begin{proof} 
Note that for every $x\in \X$ and $n\in \N$ we have $(x^\alpha, n^\beta) \in \D(\L')$ by Lemma~\ref{internal.prop.par}(a). This fact together with (SD2) guarantees that the mapping $\varphi$ is well defined. It's also clear that $\varphi$ is bijective. By Lemma~\ref{L:PartialIso}, $\alpha^{-1}$ and $\beta^{-1}$ are isomorphisms of partial groups, and $\phi$ is an isomorphism if $\phi$ and $\phi^{-1}$ are homomorphisms or partial groups. Note that $\phi^{-1}$ is the map
\[\phi^{-1}\colon\L'\longrightarrow \L,\;\Pi(x',n')\mapsto \Pi((x')^{\alpha^{-1}},(n')^{\beta^{-1}}).\]
Moreover, if $x'\in\X'$ and $n'\in\N'$, then $x=(x')^{\alpha^{-1}}\in\X$ and $n=(n')^{\beta^{-1}}\in\N$, so by assumption $(n^x)^\beta=(n^\beta)^{x^\alpha}=(n')^{x'}$. Hence,
\[((n')^{\beta^{-1}})^{(x')^{\alpha^{-1}}}=n^x=((n')^{x'})^{\beta^{-1}}.\]
So it is enough to show that $\phi$ is a homomorphism of partial groups, as it will then follow similarly with $\alpha^{-1}$, $\beta^{-1}$ and $\phi^{-1}$ in the roles of $\alpha$, $\beta$ and $\phi$ that $\phi^{-1}$ is a homomorphism of partial groups. To prove that $\phi$ is a homomorphism of partial groups, let 
\[w= (\Pi(x_1,n_1), \dots, \Pi(x_k,n_k) )  \in \D(\L)\]
with $x_1,\dots,x_k\in\X$ and $n_1,\dots,n_k\in\N$. Recall from Definition~\ref{internal.par} that 
\[w_\X=(x_1,\dots,x_k)\mbox{ and }w_\N:=( n_1^{\Pi(x_2, \dots, x_k)}, \dots, n_{k-1}^{x_k}, n_k).\]
Here, as $w\in\D(\L)$, $w_\X$ is an element of $\D(\L)$ by (SD3) and thus $w_\N$ is well-defined. Moreover, again by (SD3), we have  $w_\N\in\D(\L)$. As $w_\X\in\D(\L)\cap\W(\X)=\D(\X)$, we have $(x_{i+1},x_{i+2},\dots,x_k)\in\D(\X)$ for $i=1,\dots,k-1$ and thus, as $\alpha$ is a homomorphism of partial groups, $(x_{i+1}^\alpha,x_{i+2}^\alpha,\dots,x_k^\alpha)\in\D(\X')$. Define
\[v := ( (n_1^\beta)^{\Pi'(x_2^\alpha, \dots, x_k^\alpha)}, \dots, (n_{k-1}^\beta)^{x_k^\alpha}, n_k^\beta) \in \W(\N').\]

\smallskip

\emph{Step~1:} We show that $v =(w_\N)^\beta$. By assumption, for every $n\in\N$ and $x\in \X$, we have $(n^x)^\beta=(n^\beta)^{x^\alpha}$. Moreover, for every $i=1,\dots,k-1$, we have $\Pi'(x_{i+1}^\alpha,x_{i+2}^\alpha,\dots,x_k^\alpha)=\Pi(x_{i+1},x_{i+2},\dots,x_k)^\alpha$ as $w_\X\in\D(\X)$ and $\alpha$ is a homomorphism of partial groups. Hence, it follows that $(n_i^{\Pi(x_{i+1},\dots,x_k)})^\beta=(n_i^\beta)^{\Pi(x_{i+1},x_{i+2},\dots,x_k)^\alpha}=(n_i^\beta)^{\Pi'(x_{i+1}^\alpha,x_{i+2}^\alpha,\dots,x_k^\alpha)}$ for all $i=1,\dots,k-1$. This implies $(w_\N)^\beta=v$.

\smallskip

\emph{Step~2:} We show $(w^\phi)_{\N'}=v=(w_\N)^\beta$, $(w^\phi)_{\X'}=(w_\X)^\alpha$ and $w^\phi\in\D(\L')$. As $w\in\D(\L)$ was arbitrary, this shows $\D(\L)^\phi\subseteq\D(\L')$.

\smallskip

Recall first that, by Step~1, we have $v=(w_\N)^\beta$. Moreover, as remarked above, $w_\N \in \D(\L)$ and thus $w_\N\in\D(\L)\cap\W(\N)=\D(\N)$. Since $\beta$ is a homomorphism of partial groups we deduce that $v \in \D(\N')=\D(\L')\cap\W(\N')$. Note that $w^\phi=  (\Pi'(x_1^\alpha,n_1^\beta), \dots, \Pi'(x_k^\alpha,n_k^\beta) )$ and so (using again the notation introduced in Definition~\ref{internal.par}) $(w^\phi)_{\N'}=v\in\D(\L')$. Observe also that $(w^\phi)_{\X'}=(x_1^\alpha,\dots,x_k^\alpha)=(x_1,\dots,x_k)^\alpha=(w_\X)^\alpha\in\D(\X')\subseteq\D(\L')$ as $w_\X\in\D(\X)$ and $\alpha\colon \X\rightarrow \X'$ is a homomorphism of partial groups. So using (SD3), we conclude that $w^\phi\in\D(\L')$.

\smallskip

\emph{Step~3:} We complete the proof that $\phi$ is a homomorphism of partial groups by showing that $\Pi'(w^\phi)=\Pi(w)^\phi$. By Step~2, we have $(w^\phi)_{\X'}=(w_\X)^\alpha$ and $(w^\phi)_{\N'}=v=(w_\N)^\beta$. We conclude 
\begin{eqnarray*}
\Pi'(w^\varphi) &=& \Pi'(\Pi'((w^\phi)_{\X'}),\Pi'((w^\phi)_{\N'}))\\
&=& \Pi'(\Pi'((w_\X)^\alpha),\Pi'((w_\N)^\beta))\\
&=& \Pi'(\Pi(w_\X)^\alpha,\Pi(w_\N)^\beta)\\
&=& \Pi(\Pi(w_\X),\Pi(w_\N))^\phi\\
&=& \Pi(w)^\phi,
\end{eqnarray*}
where the first and last equality use that (SD3) holds in $\L'$ and $\L$ respectively, and the third equality uses that $\alpha$ and $\beta$ are homomorphisms of partial groups. So $\Pi(w)^\varphi =\Pi'(w^\varphi)$.
As $w$ was arbitrary, this completes the proof that $\varphi$ is an isomorphism of partial groups.
\end{proof}

\begin{theorem}\label{int.ext.iso.par}
Let $\L$ be the internal semidirect product of $\X$ with $\N$. Then $\L$ is isomorphic to $\L'=\X \ltimes_\varphi \N$, where the action $\varphi\colon \X\rightarrow\Aut(\N),x\mapsto c_x$ is the map defined in Lemma~\ref{L:InternalExternalAction}.
\end{theorem}

\begin{proof}
By Theorem~\ref{external.internal.par}, the partial group $\L'$  is the internal semidirect product of $(\X,1)$ with $(1,\N)$. Consider the natural embeddings  $\alpha \colon \X \hookrightarrow (\X,1)$ and $\beta \colon \N \hookrightarrow (1,\N)$ which are by Lemma~\ref{subgroups}(d) isomorphisms of partial groups. If $n\in \N$ and $x\in \X$ then $(n^\beta)^{x^\alpha} = (1,n)^{(x,1)} = (1, n^x) = (n^x)^\beta$, where the second equality uses the formula given in Theorem~\ref{external.internal.par} and the fact that $n^{x^\phi}=n^x$ by definition of $\phi$. Therefore, by Lemma \ref{iso.par}, we conclude that $\L$ and $\L'$ are isomorphic as partial groups.
\end{proof}

The following lemma can be seen as a version of Lemma~\ref{subgroups}(a) for internal semidirect products.

\begin{lemma}\label{L:subgroupsInternal}
Let $\L$ be a partial group which is the internal semidirect product of a partial subgroup $\X$ with a partial subgroup $\N$. Let $\Y$ be a partial subgroup of $\X$ and $\M$ be a partial subgroup of $\N$ such that $\Y\subseteq\rN_\X(\M)$. Then $\Y\M$ is a partial subgroup of $\L$.
\end{lemma}

\begin{proof}
If $w=(f_1,\dots,f_k)\in\W(\Y\M)\cap\D(\L)$ then for all $i=1,\dots,k$, we can write $f_i=\Pi(y_i,m_i)$ for some $y_i\in\Y$ and some $m_i\in\M$. By (SD3), $w_\X=(y_1,\dots,y_k)\in\D(\L)$ and $w_\N=(m_1^{\Pi(y_2,\dots,y_k)},m_2^{\Pi(y_3,\dots,y_k)},\dots,m_{k-1}^{y_k},m_k)\in\D(\L)$. Note that $w_\X\in\W(\Y)$ and $w_\N\in\W(\M)$ as $\Y$ is a partial subgroup of $\X$ normalizing $\M$. Since $\Y$ and $\M$ are partial subgroups, it follows $\Pi(w_\X)\in\Y$ and $\Pi(w_\N)\in\M$. So by (SD3), we have $\Pi(w)=\Pi(\Pi(w_\X),\Pi(w_\N))\in\Y\M$. 
\end{proof}

We next state some calculation rules for computing in internal semidirect products of partial groups.

\begin{lemma}\label{L:Conjugate}
Let $\L$ be a partial group which is an internal semidirect product of a partial subgroup $\X$ with a partial subgroup $\N$. Fix $x,y\in\X$ and $m,n\in\N$. Assume $y\in \C_\X(\N)$. Then the following hold:
\begin{itemize}
\item [(a)] We have $ym\in\D(xn)$ if and only if $y\in \D(x)$ and $m^x\in\D(n)$. Moreover, if so, then $(ym)^{xn}=y^x(m^x)^n$. 
\item [(b)] We have $ym\in\D(nx)$ if and only if $y\in\D(x)$ and $m\in\D(n)$. Moreover, if so, then $(ym)^{nx}=y^x(m^n)^x$.
\end{itemize}
\end{lemma}

\begin{proof}
Throughout this proof, we will use without further reference that, by Lemma~\ref{L:InternalExternalAction}, $(a^{x_1})^{x_2}=a^{x_1x_2}$ for all $a\in\N$ and $x_1,x_2\in\X$. Set $g:=ym$.

\smallskip

For the proof of (a) set $f=xn$. Note that, by Lemma~\ref{internal.prop.par}(c), we have $f^{-1}=\Pi(x^{-1},(n^{-1})^{x^{-1}})$. Set $u:=(f^{-1},g,f)$. By (SD3), we have $u\in\D(\L)$ if and only if $u_\X=(x^{-1},y,x)\in\D(\L)$ and  $u_\N=((n^{-1})^{x^{-1}yx},m^x,n)=((n^{-1})^{y^x},m^x,n)\in\D(\L)$; moreover, if so, then $g^f=\Pi(u)=\Pi(\Pi(u_\X),\Pi(u_\N))$. As $y\in \C_\X(\N)$, it follows from Corollary~\ref{C:CXNnormalX} that $y^x\in \C_\X(\N)$ and hence (if $u_\X\in\D(\L)$ and thus $u_\N$ is well-defined), we have $u_\N=(n^{-1},m^x,n)$. Thus, $g\in\D(f)$ if and only if $y\in\D(x)$ and $m^x\in\D(n)$. Moreover, if this is the case, then $g^f=\Pi(u)=\Pi(\Pi(u_\X),\Pi(u_\N))=\Pi(y^x,(m^x)^n)$. This shows (a).

\smallskip

For the proof of (b) set $h=nx$. Then by Lemma~\ref{internal.prop.par}(c), $h=xn^x$. Moreover, by Lemma~\ref{L:PartialGroup}(g), we have $h^{-1}=x^{-1}n^{-1}$. Set $v:=(h^{-1},g,h)$. By (SD3), we have $v\in\D(\L)$ if and only if $v_\X=(x^{-1},y,x)\in\D(\L)$ and  $u_\N=((n^{-1})^{yx},m^x,n^x)\in\D(\L)$; moreover, if so, then $g^h=\Pi(u)=\Pi(\Pi(u_\X),\Pi(u_\N))$. 

\smallskip

Suppose for a moment that $v_\X\in\D(\L)$ so that $v_\N$ is defined. Then we have $v_\N=((n^{-1})^x,m^x,n^x)$ as $y\in \C_\X(\N)$. Moreover, as $c_x$ is an automorphism of $\N$ by Lemma~\ref{L:InternalExternalAction}, we have $v_\N\in\D(\L)$ if and only if $(n^{-1},m,n)\in\D(\L)$; if so, then $\Pi(v_\N)=\Pi((n^{-1},m,n)c_x^*)=\Pi(n^{-1},m,n)c_x=(m^n)^x$.   

\smallskip

Putting everything together, we have $v\in\D(\L)$ if and only if $v_\X=(x^{-1},y,x)\in\D(\L)$ and $(n^{-1},m,n)\in\D(\L)$. Moreover, if so then $g^h=\Pi(v)=\Pi(\Pi(v_\X),\Pi(v_\N))=\Pi(y^x,(m^n)^x)$. This implies (b).
\end{proof}

\begin{cor}\label{C:CXNormalL}
Let $\L$ be a partial group which is the internal semidirect product of a partial subgroup $\X$ with a partial subgroup $\N$. If $\M\subseteq\N$ is a partial normal subgroup of $\L$, then $\C_\X(\N)\M$ is a partial normal subgroup of $\L$. In particular, $\C_\X(\N)$ and $\C_\X(\N)\N$ are partial normal subgroups of $\L$.
\end{cor}

\begin{proof}
By Corollary~\ref{C:CXNnormalX}, $\C_\X(\N)$ is a partial normal subgroup of $\X$. In particular, $\C_\X(\N)$ is a partial subgroup with $\C_\X(\N)\subseteq \rN_\X(\M)$. So by Lemma~\ref{L:subgroupsInternal}, $\C_\X(\N)\M$ is a partial subgroup. If $x\in\X$, $y\in \C_\X(\N)$, $n\in\N$ and $m\in\M$ with $ym\in\D(xn)$, then by Lemma~\ref{L:Conjugate}(a), $y\in \D(x)$, $m^x\in\D(n)$ and $(ym)^{xn}=y^x(m^x)^n$. As $y\in\C_\X(\N)\unlhd\X$, we have $y^x\in\C_\X(\N)$. Since $\M$ is a partial normal subgroup of $\L$, we have $(m^x)^n\in\M$. Hence, $(ym)^{xn}\in\C_\X(\N)\M$. This proves that $\C_\X(\N)\M$ is a partial normal subgroup. Applying this property for $\M=\{\One\}$ and $\M=\N$ (and using Lemma~\ref{internal.prop.par}(d) in the latter case), we conclude that $\C_\X(\N)$ and $\C_\X(\N)\N$ are partial normal subgroups of $\L$. 
\end{proof}

\begin{cor}\label{C:Conjnx}
Suppose $\L$ is the internal semidirect product of a partial subgroup $\X$ with a partial subgroup $\N$. Let $n,m\in\N$ and $x\in\X$. 
\begin{itemize}
 \item [(a)] We have $m\in\D(xn)$ if and only if $m^x\in\D(n)$. If so, then $(m^x)^n=m^{xn}$.
 \item [(b)] We have $m\in\D(n)$ if and only if $m\in\D(nx)$. If so, then $(m^n)^x=m^{nx}$.
 \item [(c)] We have $m^{x^{-1}}\in\D(n)$ if and only if $m\in\D(n^x)$. If so, then $m^{n^x}=((m^{x^{-1}})^n)^x$.
\end{itemize}
\end{cor}

\begin{proof}
Note that $\One\in\D(x)$ with $\One^x=\One$. So properties (a) and (b) follow from Lemma~\ref{L:Conjugate}(a),(b) applied with $y=\One\in \C_\X(\N)$.

\smallskip

As $\X\subseteq N_\L(\N)$, Lemma~\ref{L:PartialGroup}(b) gives $(x^{-1},n,x,x^{-1})\in\D(\L)$. Thus, by (PG3), $n^xx^{-1}=\Pi(x^{-1},n,x,x^{-1})=x^{-1}n$. By (a), we have $m^{x^{-1}}\in\D(n)$ if and only if $m\in\D(x^{-1}n)=\D(n^xx^{-1})$. By (b), this is the case if and only if $m\in\D(n^x)$. Moreover, if these equivalent conditions hold, we have $(m^{n^x})^{x^{-1}}\overset{(b)}{=}m^{n^xx^{-1}}=m^{x^{-1}n}\overset{(a)}{=}(m^{x^{-1}})^n$. Conjugating this equation with $x$ and using Lemma~\ref{L:InternalExternalAction}, one obtains $m^{n^x}=((m^{x^{-1}})^n)^x$. This proves (c).
\end{proof}

In the next section, we will prove that, under certain sufficient conditions, we can construct a locality structure on partial groups which are internal or external semidirect products. The following lemma is a crucial preliminary step.

\begin{lemma}\label{L:SylowProduct}
Let $\L$ be the internal semidirect product of a partial subgroup $\X$ with a partial subgroup $\N$. Suppose $S_\X$ is a subgroup of $\X$ and $T$ is a subgroup of $\N$ such that $S_\X\subseteq N_\L(T)$.
\begin{itemize}
 \item [(a)] $S_\X T:=\{\Pi(s,t)\colon s\in S_\X\mbox{ and }t\in T\}$ is a subgroup of $\L$ which (as a binary group) is the semidirect product of $S_\X$ with $T$ in the usual group theoretical sense.
 \item [(b)] If $S_\X$ and $T$ are $p$-subgroups, then $S_\X T$ is a $p$-subgroup.
 \item [(c)] If $S_\X$ is a maximal $p$-subgroup of $\X$ and $T$ is a maximal $p$-subgroup of $\N$, then $S_\X T$ is a maximal $p$-subgroup of $\L$.
\end{itemize}
\end{lemma}

\begin{proof}
Let $w\in\W(S_\X T)$. Then every entry of $w$ is of the form $\Pi(s,t)$ with $s\in S_\X$ and $t\in T$. Hence, $w_\X\in\W(S)\subseteq\D(\L)$ with $\Pi(w_\X)\in S_\X$, as $S_\X$ is a subgroup of $\X$. In particular, $w_\N$ is defined. As $S_\X\subseteq N_\L(T)$, one sees easily that $w_\N\in\W(T)$. So as $T$ is a subgroup of $\N$, we have $w_\N\in\D(\L)$ and $\Pi(w_\N)\in T$. Hence, by (SD3), $w\in\D(\L)$ and $\Pi(w)=\Pi(\Pi(w_\X),\Pi(w_\N))\in S_\X T$. This proves that $S_\X T$ is a subgroup of $\L$. As $S_\X$ and $T$ are clearly both contained in $S_\X T$, it follows now that $S_\X T$ is also the product of its subgroups $S_\X$ and $T$ in the usual group theoretical sense. Furthermore, as $S_\X\subseteq N_\L(T)$, we have that $T$ is normal in $S_\X T$. By Lemma~\ref{internal.prop.par}(b), we have $S_\X\cap T\subseteq \X\cap \N=\{\One\}$. This shows (a), and property (b) follows directly from (a).

\smallskip

For the proof of (c) assume now that $S_\X$ is a maximal $p$-subgroup of $\X$ and $T$ is a maximal $p$-subgroup of $\N$. Let furthermore $S_\X T\subseteq P$ for some $p$-subgroup $P$ of $\L$. By (a), it is enough to show that $S_\X T=P$. Set
\[P_\X:=\{x\in\X\colon \exists n\in\N\mbox{ such that }(x,n)\in\D\mbox{ and }\Pi(x,n)\in P\}.\]

\smallskip

\noindent\emph{Step~1:} We show that $P_\X$ is a subgroup of $\X$. Let $u=(x_1,\dots,x_k)\in\W(P_\X)$. Then by definition of $P_\X$, there exist $n_1,\dots,n_k\in\N$ with $\Pi(x_i,n_i)\in P$ for $i=1,\dots,k$. Hence, $w:=(\Pi(x_1,n_1),\dots,\Pi(x_k,n_k))\in\W(P)\subseteq\D(\L)$ with $\Pi(w)\in P$, as $P$ is a subgroup of $\L$. Thus, by (SD3), $u=w_\X\in\D(\L)$, $w_\N\in\D(\L)$ and $\Pi(\Pi(u),\Pi(w_\N))=\Pi(\Pi(w_\X),\Pi(w_\N))=\Pi(w)\in P$ with $\Pi(w_\N)\in\N$. So by definition of $P_\X$, we have $\Pi(u)\in P_\X$. 

\smallskip

\noindent\emph{Step~2:} We show that $P_\X=S_\X$. To see this note first that $S_\X\subseteq P_\X$, since for every $s\in S_\X$, we have $(s,\One)\in\D(\L)$ and $\Pi(s,\One)\in S_\X T\subseteq P$ with $\One\in\N$. Hence, as $S_\X$ is a maximal $p$-subgroup of $\X$, it is enough to show that $P_\X$ is a $p$-subgroup of $\X$. By Step~1, $P_\X$ is a subgroup of $\X$. Define now 
\[\phi\colon P\rightarrow P_\X,\Pi(x,n)\mapsto x\mbox{ for all }x\in\X,\;n\in\N\mbox{ with }\Pi(x,n)\in P.\]
Notice that $\phi$ is well-defined by (SD2) and surjective by definition of $P_\X$. We show now that $\phi$ is a homomorphism of groups. For that let $f_1,f_2\in P$ and write $f_i=\Pi(x_i,n_i)$ with $x_i\in\X$ and $n_i\in\N$ for $i=1,2$. As $P$ is a subgroup, $v:=(f_1,f_2)\in\D(\L)$. So by (SD3), $v_\X=(x_1,x_2)=\D(\L)$, $v_\N$ is well-defined and an element of $\W(\N)\cap\D(\L)$, and $f_1f_2=\Pi(v)=\Pi(x_1x_2,\Pi(v_\N))$. Hence, $\phi(f_1f_2)=x_1x_2=\phi(f_1)\phi(f_2)$. So $\phi$ is a surjective group homomorphism. As $P$ is a $p$-group, it follows that $P_\X$ is a $p$-group as well. As argued above, this yields that $S_\X=P_\X$.

\smallskip

\noindent\emph{Step~3:} We complete the proof. Observe that by (SD2) and definition of $P_\X$, we have $P\subseteq P_\X\N$. So using Step~2, we conclude $S_\X\subseteq S_\X T\subseteq P\subseteq S_\X\N$. Hence, by the Dedekind Lemma \cite[2.1]{Henke:2015a}, we have $P=S_\X(P\cap\N)$. Notice that $T\subseteq P\cap\N$. Moreover, $P\cap\N$ is a subgroup of the $p$-group $P$, and thus a $p$-subgroup of $\N$. As $T$ is a maximal $p$-subgroup of $\N$, it follows thus that $T=P\cap\N$. Hence, $P=S_\X T$ as required. 
\end{proof}

\section{Semidirect products of localities}\label{S:SemidirectProdLoc}

In the next subsection, we will show that, under certain sufficient conditions which are made precise in Hypothesis~\ref{H:ProdLocInternal} below, we can endow a partial group, which is an internal semidirect product of two partial subgroups, with a locality structure. This will motivate definitions of internal and external semidirect products of localities which we give in Subsection~\ref{Ss:SemidirectLocalitiesDef}. Moreover, given that an external semidirect product of two partial groups is by Theorem~\ref{external.internal.par} also an internal semidirect product of partial subgroups, the results we prove in Subsection~\ref{Ss:ConstructLocalities} will imply that external semidirect products of localities (as we will define them) form indeed localities.

\subsection{Constructing localities}\label{Ss:ConstructLocalities}

\begin{hypo}\label{H:ProdLocInternal}
Let $\L$ be a partial group which is the internal semidirect product of a partial subgroup $\X$ with a partial subgroup $\N$. Assume that for appropriate $S_\X$, $T$, $\Delta_\X$ and $\Gamma$, the triples $(\X,\Delta_\X,S_\X)$ and $(\N,\Gamma,T)$ form localities and $\X$ leaves $\Gamma$ invariant, i.e. $R^x\in\Gamma$ for every $R\in\Gamma$ and $x\in\X$. Suppose furthermore
\begin{equation}\tag{$*$}\label{star}
 Q\cap \C_\X(\N)\in\Delta_\X\mbox{ for all }Q\in\Delta_\X.
\end{equation}
Set $S:=S_\X T$ and notice that $S$ is a $p$-subgroup of $\L$ by Lemma~\ref{L:SylowProduct}. For $P\leq S$ define
\[P_\X:=\{s\in \C_{S_\X}(\N)\colon \exists t\in T\mbox{ such that }st\in P\}\]
and 
\[P_\N:=\{t\in T\colon \exists s\in \C_{S_\X}(\N)\mbox{ such that }st\in P\}.\]
Define the following sets of subgroups of $S$:
\begin{itemize}
 \item Write $\Delta_0$ for the set of all subgroups of $S$ of the form $QR$ where $Q\in\Delta_\X$, $R\in\Gamma$ and $Q\subseteq \C_\X(\N)$.
 \item Write $\Delta$ for the set of all subgroups $P$ of $S$ which contain an element of $\Delta_0$.
 \item Write $\Delta^+$ for the set of all subgroups $P$ of $S$ such that $P_\X\in\Delta_\X$ and $P_\N\in\Gamma$.
\end{itemize}
\end{hypo}

\begin{lemma}\label{L:RemarkProdLocInternal}
Assume Hypothesis \ref{H:ProdLocInternal} and let $P\leq S$. Then the following hold:
\begin{itemize}
\item[(a)] $P \cap \C_{S_\X}(\N) \leq P_\X$ and $(P\cap \X)_\X = P \cap \C_{S_\X}(\N)$.
\item[(b)] $P \cap \N \leq P_\N$ and $(P\cap \N)_\N = P \cap \N$.
\item[(c)] $S_0:=\C_{S_\X}(\N)T=(\C_\X(\N)\N)\cap S$ is strongly closed in $\F_S(\L)$. Moreover, $(P\cap S_0)_\X=P_\X$ and $(P\cap S_0)_\N=P_\N$. In particular, $P\in\Delta^+$ if and only if $P\cap S_0\in\Delta^+$. 
\end{itemize}
\end{lemma}

\begin{proof}
Properties (a) and (b) are easy to check. For (c) note first that, by (SD2), every element of $S=S_\X T$ can be written uniquely as a product of an element of $\X$ with an element of $\N$. This implies $(\C_\X(\N)\N)\cap S=\C_{S_\X}(\N)T=S_0$. By Lemma~\ref{C:CXNormalL}, $\C_\X(\N)\N$ is normal in $\L$, and from that one sees easily that $S_0$ is strongly closed in $\F_S(\L)$. One verifies directly from the definitions of $P_\X$ and $P_\N$ that the second part of (c) holds.
\end{proof}

\begin{lemma}\label{L:Delta0Closed}
Assume Hypothesis~\ref{H:ProdLocInternal}. Let $R\in\Gamma$ and $Q\in\Delta_\X$ such that $Q\subseteq \C_\X(\N)$. Let $f=\Pi(x,n)\in\L$ with $x\in\X$ and $n\in\N$. Then $QR\subseteq S_f$ if and only if $Q\subseteq S_x$ and $R^x\subseteq S_n$. Moreover, if so, then we have $Q^x\in\Delta_\X$, $(R^x)^n\in\Gamma$, $Q^x\subseteq \C_\X(\N)$, and $(QR)^f=Q^x(R^x)^n\in\Delta_0$. In particular, $\Delta_0$ is closed under taking $\L$-conjugates in $S$.
\end{lemma}

\begin{proof}
For the proof observe first that $R^x\in\Gamma$ as $\Gamma$ is $\X$-invariant. In particular, $R^x\leq T\leq S$. Moreover, applying Lemma~\ref{L:Conjugate}(a) for all $y\in Q\subseteq \C_\X(\N)$ and all $m\in R$, one sees that $QR\subseteq \D(f)$ if and only if $Q\subseteq \D(x)$ and $R^x\subseteq\D(n)$, and if so, then $(QR)^f=Q^x(R^x)^n$. Hence, $QR\subseteq S_f$ if and only if $Q\subseteq S_x$ and $R^x\subseteq S_n$. As $S_\X$ and $T$ are maximal $p$-subgroups of $\X$ and $\N$ respectively, we have $S\cap \X=S_\X$ and $S\cap\N=T$. So if $Q\leq S_x$, then $Q\leq (S_\X)_x$, and if $R^x\leq S_n$, then $R^x\leq T_n$. Thus, as $(\X,\Delta_\X,S_\X)$ and $(\N,\Gamma,T)$ are localities, it follows in this case $Q^x\in\Delta_\X$, $(R^x)^n\in\Gamma$. Moreover, by Corollary~\ref{C:CXNnormalX}, we have $Q^x\subseteq \C_\X(\N)$. So $(QR)^f=Q^x(R^x)^n\in\Delta_0$. As every element of $\L$ is of the form $\Pi(x,n)$ for some $x\in\X$ and $n\in\N$, it follows that $\Delta_0$ is $\L$-closed in $S$. 
\end{proof}

\begin{lemma}\label{L:DDelta0}
Assume Hypothesis~\ref{H:ProdLocInternal}. Then $\D(\L)=\D(\L)_{\Delta_0}$.
\end{lemma}

\begin{proof}
Let $w=(f_1,\dots,f_k)\in\W(\L)$ and write $f_i=\Pi(x_i,n_i)$ with $x_i\in\X$ and $n_i\in\N$. Then $w_\X=(x_1,\dots,x_k)$. If $w_\X\in\D(\L)$, set $y_i:=\Pi(x_{i+1},\dots,x_k)$ for all $0,1,\dots,k$ (meaning $y_k=\Pi(\emptyset)=\One)$, and $\ov{n}_i=n_i^{y_i}$ for all $i=1,\dots,k$; note that $y_0,y_1,\dots,y_k$, $\ov{n}_1,\dots,\ov{n}_k$ and $w_\N:=(\ov{n}_1,\dots,\ov{n}_k)$ are well-defined in this case.

\smallskip

\emph{Step~1:} We show that $\D(\L)\subseteq\D(\L)_{\Delta_0}$. 

\smallskip

For the proof assume $w\in\D(\L)$. Then by (SD3), $w_\X=(x_1,\dots,x_k)\in\D(\L)\cap\W(\X)=\D(\X)$ and $w_\N=(\ov{n}_1,\dots,\ov{n}_k)$ is well-defined and an element of $\D(\L)\cap\W(\N)=\D(\N)$. As $(\X,\Delta_\X,S_X)$ and $(\N,\Gamma,T)$ are localities, it follows that there exist $Q_0,\dots,Q_k\in\Delta_\X$ and $R_0,\dots,R_k\in\Gamma$ such that $Q_{i-1}\subseteq\D(x_i)$, $Q_{i-1}^{x_i}=Q_i$, $R_{i-1}\subseteq\D(\ov{n}_i)$ and $R_{i-1}^{\ov{n}_i}=R_i$ for $i=1,\dots,k$. By (\ref{star}) and Corollary~\ref{C:CXNnormalX}, replacing $Q_i$ by $Q_i\cap \C_\X(\N)$, we may assume $Q_i\leq \C_\X(\N)$ for all $i=1,\dots,k$. For $i=0,1,\dots,k$ set 
\[\ov{R}_i:=R_i^{{y_i}^{-1}}\mbox{ and }P_i=Q_i\ov{R}_i.\]
Notice that, as $\X$ acts on $\Gamma$, we have $\ov{R}_i\in\Gamma$ and so $P_i\in\Delta_0$ for $i=0,1,\dots,n$. Let now $i\in\{1,\dots,k\}$. As $R_{i-1}\subseteq \D(\ov{n}_i)=\D(n_i^{y_i})$, it follows from Corollary~\ref{C:Conjnx}(c) that $R_{i-1}^{y_i^{-1}}\subseteq\D(n_i)$ and
\begin{eqnarray}\label{2}
((R_{i-1}^{y_i^{-1}})^{n_i})^{y_i}=R_{i-1}^{\ov{n}_i}=R_i.
\end{eqnarray}
Using $y_i^{-1}=\Pi(x_k^{-1},\dots,x_{i+1}^{-1})=\Pi(x_k^{-1},\dots,x_{i+1}^{-1},x_i^{-1},x_i)=\Pi(x_k^{-1},\dots,x_{i}^{-1})x_i=y_{i-1}^{-1}x_i$ and Lemma~\ref{L:InternalExternalAction}, one sees now that $\ov{R}_{i-1}^{x_i}=R_{i-1}^{y_{i-1}^{-1}x_i}=R_{i-1}^{y_i^{-1}}\subseteq\D(n_i)$. So conjugating (\ref{2}) with $y_i^{-1}$ and using Lemma~\ref{L:InternalExternalAction} again, it follows 
\[(\ov{R}_{i-1}^{x_i})^{n_i}=(R_{i-1}^{y_i^{-1}})^{n_i}=R_i^{y_i^{-1}}=\ov{R}_i.\]
In particular, $\ov{R}_{i-1}^{x_i}\subseteq S_{n_i}$. By the choice of $Q_0,Q_1,\dots,Q_k$, we have $Q_{i-1}\leq S_{x_i}$, $Q_{i-1}^{x_i}=Q_i$ and $Q_{i-1}\subseteq \C_\X(\N)$. So by Lemma~\ref{L:Delta0Closed}, we have $P_{i-1}\leq S_{f_i}$ and 
\[P_{i-1}^{f_i}=(Q_{i-1}\ov{R}_{i-1})^{f_i}=Q_{i-1}^{x_i}(\ov{R}_{i-1}^{x_i})^{n_i}=Q_i\ov{R}_i=P_i.\] 
Since $i\in\{1,\dots,k\}$ was arbitrary, this shows that $w=(f_1,\dots,f_k)\in\D(\L)_{\Delta}$ via $P_0,P_1,\dots,P_k$. This completes Step~1.

\smallskip

\emph{Step~2:} We show $\D(\L)_{\Delta_0}\subseteq\D(\L)$.

\smallskip

For the proof assume $w\in\D(\L)_{\Delta_0}$ via $P_0,P_1,\dots,P_k\in\Delta_0$. Write $P_0=Q_0\ov{R}_0$ where $Q_0\in\Delta_\X$ and $\ov{R}_0\in\Gamma$ with $Q_0\subseteq \C_\X(\N)$. For $i=1,\dots,k$ set $Q_i:=Q_{i-1}^{x_i}$ and $\ov{R}_i:=(\ov{R}_{i-1}^{x_i})^{n_i}$. Using induction on $i$, it follows from Lemma~\ref{L:Delta0Closed} that for $i=0,1,\dots,k$, $Q_i$ and $\ov{R}_i$ are well-defined (i.e. if $i\geq 1$, $Q_{i-1}\subseteq \D(x_i)$ and $R_{i-1}^{x_i}\subseteq\D(n_i)$), $P_i=Q_i\ov{R}_i$, $Q_i\in \Delta_\X$, $\ov{R}_i\in\Gamma$ and $Q_i\subseteq \C_\X(\N)$. As $(\X,\Delta_\X,S_\X)$ is a locality, we can in particular conclude that $w_\X=(x_1,\dots,x_k)\in\D(\X)=\D(\L)\cap\W(\X)$ via $Q_0,Q_1,\dots,Q_k$. So $y_0,y_1,\dots,y_k$, $\ov{n}_1,\dots,\ov{n}_k$ and $w_\N=(\ov{n}_1,\dots,\ov{n}_k)$ are well-defined. As $\Gamma$ is $\X$-invariant, we have for all $i=0,1,\dots,k$ that $R_i:=\ov{R}_i^{y_i}\in\Gamma$. We will show that $w_\N\in\D(\N)=\D(\L)\cap\W(\N)$ via $R_0,R_1,\dots,R_k$. For that fix $i\in\{1,\dots,k\}$. Since $y_i^{-1}=y_{i-1}^{-1}x_i$ (as seen in Step~1), we have  $R_{i-1}^{y_i^{-1}}=(R_{i-1}^{y_{i-1}^{-1}})^{x_i}=\ov{R}_{i-1}^{x_i}\subseteq\D(n_i)$. Hence, Corollary~\ref{C:Conjnx}(c) gives $R_{i-1}\subseteq\D(n_i^{y_i})=\D(\ov{n}_i)$ and $R_{i-1}^{\ov{n}_i}=((R_{i-1}^{y_i^{-1}})^{n_i})^{y_i}=((\ov{R}_{i-1}^{x_i})^{n_i})^{y_i}=\ov{R}_i^{y_i}=R_i$. As $i$ was arbitrary and $(\N,\Gamma,T)$ is a locality, this shows that $w_\N=(\ov{n}_1,\dots,\ov{n}_k)\in\D(\N)=\D(\L)\cap\W(\N)$ via $R_0,R_1,\dots,R_k$. So we have shown that $w_\X\in\D(\L)$ and $w_\N\in\D(\L)$. Hence, by (SD3), it follows $w\in\D(\L)$. This completes the proof of Step~2. Together with Step~1, this yields the assertion.
\end{proof}

\begin{lemma}\label{L:ProdLocInternal}
Assume Hypothesis~\ref{H:ProdLocInternal}. Then $(\L,\Delta,S)$ and $(\L,\Delta^+,S)$ are localities.
\end{lemma}

\begin{proof}
As $(\X,\Delta_\X,S_\X)$ and $(\N,\Gamma,T)$ are localities, $S_\X$ and $T$ are maximal $p$-subgroups of $\X$ and $\N$ respectively. Moreover, as $\X$ leaves $\Gamma$ invariant and $T$ is the unique maximal member of $\Gamma$, we have $S_\X\subseteq\X\subseteq N_\L(T)$. Hence, by Lemma~\ref{L:SylowProduct}, $S_\X T$ is a maximal $p$-subgroup of $\L$. It follows thus from Lemma~\ref{L:ConstructLocality}, Lemma~\ref{L:Delta0Closed} and Lemma~\ref{L:DDelta0} that $(\L,\Delta,S)$ is a locality. In particular, for every $f\in\L$, $S_f$ is a subgroup of $S$. Moreover, for all $P\leq S_f$, $P^f$ is a subgroup of $S$. 

\smallskip

We argue now that $(\L,\Delta^+,S)$ is a locality. Note that $\Delta^+$ is closed under taking overgroups in $S$, since $\Delta_\X$ and $\Gamma$ are closed under taking overgroups in $S_\X$ and $T$ respectively. It remains to show that $\Delta^+$ is closed under taking $\L$-conjugates in $S$, and that $\D(\L)=\D(\L)_{\Delta^+}$. We will prove these properties in the following two steps.

\smallskip

\emph{Step~1:} Let $P\in\Delta^+$ with $P\leq \C_{S_\X}(\N)T$ and $f=\Pi(x,n)\in\L$ with $x\in\X$ and $n\in\N$. We show that $P\leq S_f$ if and only $P_\X\leq S_x$ and $(P_\N)^x\leq S_n$. Moreover, if so, then $P^f\leq \C_{S_\X}(\N)T$, $(P^f)_\X=(P_\X)^x\in\Delta_\X$,  $(P^f)_\N=((P_\N)^x)^n\in\Gamma$ and $P^f\in\Delta^+$.

\smallskip

Essentially, this follows from Lemma~\ref{L:Conjugate}(a). Namely, given $g=st\in P$ with $s\in P_\X$ and $t\in P_\N$, it follows from this lemma that $g\in\D(f)$ if and only if $s\in\D(x)$ and $t^x\in\D(n)$; moreover, if so, then $g^f=(st)^{xn}=s^x(t^x)^n$. If $g\in \D(f)$, then note that $s^x\in\X$ and $(t^x)^n\in\N$. By (SD2), every element of $\L$ can be written uniquely as a product of an element of $\X$ and an element of $\N$. So if $f\in\D(f)$, then $g^f=s^x(t^x)^n\in S=S_\X T$ if and only if $s^x\in S_\X$ and $(t^x)^n\in T$. Since $S_\X=S\cap\X$ and $T=S\cap\N$, we can conclude altogether that $g\in S_f$ if and only if $s\in S_x$ and $t^x\in S_n$; moreover, if so then $g^f=s^x(t^x)^n$. As $P\leq \C_{S_\X}(\N)T$, it follows from the definition of $P_\X$ and $P_\N$ that every element $g\in P$ can be written as a product $g=st$ with $s\in P_\X$ and $t\in P_\N$. The other way around, for every $s\in P_\X$, there exists $t\in P_\N$ with $st\in P$, and for every $t\in P_\N$, there exists $s\in P_\X$ with $st\in P$. So what we proved implies that $P\leq S_f$ if and only if $P_\X\leq S_x$ and $(P_\N)^x\leq S_n$. Moreover, if this is the case then, since $\C_\X(\N)\unlhd \X$ by Lemma~\ref{C:CXNnormalX}, we can conclude that  $P^f\leq \C_{S_\X}(\N)T$, $(P^f)_\X=(P_\X)^x\leq \C_{S_\X}(\N)$ and $(P^f)_\N=((P_\N)^x)^n$. Since $(\X,\Delta_\X,S_\X)$ and $(\N,\Gamma,T)$ are localities and $\Gamma$ is $\X$-invariant, the claim above follows and Step~1 is complete.

\smallskip

For the next two steps, set $S_0:=\C_{S_\X}(\N)T$. We will use without further reference that, by Lemma~\ref{L:RemarkProdLocInternal}(c), $S_0$ is strongly closed and, for every $Q\leq S$, we have $Q\in\Delta^+$ if and only if $Q\cap S_0\in\Delta^+$.

\smallskip

\emph{Step~2:} We argue that $\Delta^+$ is closed under taking $\L$-conjugates.

\smallskip

Let $P\in\Delta^+$ and $f\in\L$ with $P\leq S_f$. As $P\in\Delta^+$, we have also $P\cap S_0\in\Delta^+$. Hence, by Step~1, $(P\cap S_0)^f\in\Delta^+$. As $S_0$ is strongly closed, $P^f\cap S_0=(P\cap S_0)^f$. So $P^f\in\Delta^+$ completing Step~2.

\smallskip

\emph{Step~3:} We show that $\D(\L)=\D(\L)_{\Delta^+}$. 

\smallskip

Note first that $\Delta\subseteq\Delta^+$ and thus $\D(\L)=\D(\L)_{\Delta}\subseteq\D(\L)_{\Delta^+}$. Define
\[\gamma\colon \Delta^+\rightarrow\Delta_0,\;P\mapsto P_\X P_\N.\]
Let $P,Q\in\Delta^+$ with $f=\Pi(x,n)\in\L$ with $P\subseteq S_f$ and $P^f=Q$. We want to argue that $\D(\L)_{\Delta^+}\subseteq \D(\L)_{\Delta_0}=\D(\L)$ with the latter equality using Lemma~\ref{L:DDelta0}. By Remark~\ref{R:DDelta}, it is sufficient to show that $\gamma(P)\leq S_f$ and $\gamma(P)^f=\gamma(Q)$. As $P$ and $Q$ are elements of $\Delta^+$, $P\cap S_0$ and $Q\cap S_0$ are elements of $\Delta^+$ too. Moreover, $S_0$ is strongly closed in $\F_S(\L)$ and thus $(P\cap S_0)^f=P^f\cap S_0=Q\cap S_0$. By Lemma~\ref{L:RemarkProdLocInternal}(c), we have $P_\X=(P\cap S_0)_\X$ and $P_\N=(P\cap S_0)_\N$, thus $\gamma(P)=\gamma(P\cap S_0)$ and similarly $\gamma(Q)=\gamma(Q\cap S_0)$. Therefore, replacing $P$ and $Q$ by $P\cap S_0$ and $Q\cap S_0$, we may assume that $P$ and $Q$ are contained in $S_0$. It follows then from Lemma~\ref{L:Delta0Closed} and Step~1 that $\gamma(P)\subseteq S_f$ and $\gamma(P)^f=P_\X^x(P_\N^x)^n=\gamma(P^f)=\gamma(Q)$. This completes Step~3 and thus the proof of the assertion.
\end{proof}

\begin{cor}\label{C:ProdLocInternal}
Assume Hypothesis~\ref{H:ProdLocInternal} and let $\Delta^*$ be a set of subgroups of $S$ such that $\Delta\subseteq \Delta^*\subseteq\Delta^+$ and $\Delta^*$ is closed under taking $\L$-conjugates and overgroups in $S$. Then $(\L,\Delta^*,S)$ is a locality.
\end{cor}

\begin{proof}
As $\Delta\subseteq\Delta^*\subseteq\Delta^+$, we have $\D(\L)_\Delta\subseteq\D(\L)_{\Delta^*}\subseteq\D(\L)_{\Delta^+}$. It follows from  Lemma~\ref{L:ProdLocInternal} that $S$ is a maximal $p$-subgroup of $\L$ and $\D(\L)_{\Delta}=\D(\L)=\D(\L)_{\Delta^+}$. In particular, we have that $\D(\L)_{\Delta^*}=\D(\L)$. As $\Delta^*$ is by assumption closed under taking $\L$-conjugates and overgroups in $S$, this shows that $(\L,\Delta^*,S)$ is a locality.
\end{proof}

\subsection{Internal and external semidirect products of localities}\label{Ss:SemidirectLocalitiesDef}

The results we proved in the previous subsection motivate the following definitions.

\begin{definition}[Internal semidirect product of localities]\label{D:InternalProdLoc}
Let $(\X,\Delta_\X, S_\X)$ and $(\N,\Gamma,T)$ be sublocalities of a locality $(\L,\Delta^*,S)$. We say that $(\L,\Delta^*,S)$ is an \emph{internal semidirect product} of $(\X,\Delta_\X,S_\X)$ with $(\N,\Gamma,T)$ if the following properties hold:
\begin{itemize}
\item $\L$ is the internal semidirect product of $\X$ with $\N$ and $S=S_\X T$;
\item $\X$ leaves $\Gamma$ invariant, i.e. $R^x\in\Gamma$ for all $R\in\Gamma$ and $x\in\X$;
\item $Q\cap \C_\X(\N)\in\Delta_\X\mbox{ for all }Q\in\Delta_\X$;
\item $\Delta\subseteq \Delta^*\subseteq\Delta^+$, where $\Delta$ and $\Delta^+$ are the sets of subgroups defined in Hypothesis~\ref{H:ProdLocInternal}.
\end{itemize}
More precisely, we say then that $(\L,\Delta^*,S)$ is the internal semidirect product of the locality $(\X,\Delta_\X,S_\X)$ with the locality $(\N,\Gamma,T)$. If in addition $\Delta^*=\Delta$, this internal semidirect product is said to be \emph{sparse}. If $\Delta^*=\Delta^+$, we call the internal semidirect product \emph{ample}.
\end{definition}

\begin{definition}[Actions of localities on localities]
We say that a locality $(\X,\Delta_\X,S_\X)$ \emph{acts} on a locality $(\N,\Gamma,T)$ if the partial group $\X$ acts on the partial group $\N$ via a homomorphism $\phi\colon\X\rightarrow \N$ and in addition the following two properties hold:
\begin{itemize}
 \item $\X$ acts on $\Gamma$, i.e. $R^{x^\phi}\in\Gamma$ for all $R\in\Gamma$ and $x\in\X$;
 \item we have $Q\cap \C_\X^\phi(\N)\in\Delta_\X$ for every $Q\in\Delta_\X$.
\end{itemize}
If such $\phi$ is given, we say also that $(\X,\Delta_\X,S_\X)$ acts on $(\N,\Gamma,T)$ via $\phi$.
\end{definition}

Note that, if $(\X,\Delta_\X,S_\X)$ and $(\N,\Gamma,T)$ are as in the above definition, then $\X$ acts also on $T$, since $\X$ acts on $\Gamma$ and $T$ is the unique maximal member of $\Gamma$.

\begin{definition}[External semidirect product of localities]\label{D:ExtProdLoc}
Suppose a locality $(\X, \Delta_\X, S_\X)$ acts on a locality $(\N, \Gamma,T)$ via some homomorphism $\phi\colon \X\rightarrow \Aut(\N)$. Set $\L = \X \ltimes_\varphi \N$ and $S = (S_\X, T)$ (and note that $S$ is a $p$-subgroup of $\L$ by Theorem~\ref{external.internal.par} and Lemma~\ref{L:SylowProduct}). For $P\leq S$ set 
\[P_\X^\phi:=\{s\in \C_{S_\X}^\phi(\N)\colon \exists t\in T\mbox{ such that }(s,t)\in P\}\] 
and 
\[P_\N^\phi:=\{t\in T\colon \exists s\in \C_{S_\X}^\phi(\N)\mbox{ such that }(s,t)\in P\}.\]
Define moreover
\[\Delta_\phi:=\{P\leq S\colon \exists\; Q \in \Delta_\X\;\exists R \in \Gamma\mbox{ such that }(Q,R)\leq P\}\]
and
\[\Delta^+_\phi:=\{P\leq S\colon P_\X^\phi\in\Delta_\X\mbox{ and }P_\N^\phi\in\Gamma\}.\]
Then 
\begin{itemize}
\item the \emph{sparse external semidirect product} of $(\X,\Delta_\X,S_\X)$ with $(\N,\Gamma,T)$ (via $\phi$) is the triple $(\L,\Delta_\phi,S)$;
\item the \emph{ample external semidirect product} of $(\X,\Delta_\X,S_\X)$ with $(\N,\Gamma,T)$ (via $\phi$) is the triple $(\L,\Delta^+_\phi,S)$.
\end{itemize}
More generally, \emph{external semidirect product} of $(\X,\Delta_\X,S_\X)$ with $(\N,\Gamma,T)$ (via $\phi$) is a triple $(\L,\Delta^*,S)$, where $\Delta_\phi\subseteq\Delta^*\subseteq\Delta^+_\phi$ and $\Delta^*$ is closed under taking $\L$-conjugates and overgroups in $S$. 
\end{definition}

Note that it is a priori not completely clear that our definitions makes sense, i.e. that sparse and ample external semidirect products are special cases of external semidirect products as defined at the end of Definition~\ref{D:ExtProdLoc}; this is because it is not clear that $\Delta_\phi$ and $\Delta^+_\phi$ are closed under taking $\L$-conjugates and overgroups in $S$. We will prove this however in part (a) of the following lemma. We show moreover some other crucial properties:  Every external semidirect product of localities is a locality, which is an internal semidirect product of canonically defined sublocalities.

\begin{lemma}\label{Ext.loc}
Let  $(\X, \Delta_\X, S_\X)$ be a locality acting on a locality $(\N, \Gamma,T)$ via a homomorphism $\phi$. Set $\L=\X\ltimes_\phi\N$ and $S=(S_\X,T)$. For parts (b),(c) and (d), let $(\L,\Delta^*,S)$ be an external semidirect product of $(\X, \Delta_\X, S_\X)$ with $(\N, \Gamma,T)$ via $\phi$. Then the following hold:
\begin{itemize}
\item [(a)] The sets $\Delta_\phi$ and $\Delta^+_\phi$ (as defined in Definition~\ref{D:ExtProdLoc}) are closed under taking $\L$-conjugates and overgroups in $S$; so the sparse external semidirect product and the ample external semidirect product of $(\X,\Delta_\X,S_\X)$ with $(\N,\Gamma,T)$ are indeed external semidirect products in the more general sense.
\item [(b)] The external semidirect product $(\L,\Delta^*,S)$ is a locality. 
\item [(c)] Adopting  Notation~\ref{N:ExternalCartesianProd}, set $\hat{\X}=(\X,\One)$, $\hat{S}_\X=(S_\X,\One)$, $\hat{\N}=(\One,\N)$, $\hat{T}=(\One,T)$, 
\[\hat{\Delta}_\X = \{(Q,\One) \colon Q \in \Delta_\X \} \mbox{ and } \hat{\Gamma}= \{(1,R) \colon R \in \Gamma \}.\]
Then $(\hat{\X},\hat{\Delta}_\X,\hat{S}_\X))$ and $(\hat{\N},\hat{\Gamma},\hat{T})$ are sublocalities of $(\L, \Delta^*, S)$, and $(\L,\Delta^*,S)$ is an internal semidirect product of $(\hat{\X}, \hat{\Delta}_\X, \hat{S}_\X)$ with $(\hat{\N}, \hat{\Gamma},\hat{T})$. 
\item [(d)] The external semidirect product $(\L,\Delta^*,S)$ is sparse if and only if it is a sparse internal semidirect product of $(\hat{\X}, \hat{\Delta}, \hat{S}_\X)$ with $(\hat{\N}, \hat{\Gamma},\hat{T})$. Similarly, $(\L,\Delta^*,S)$ is an ample external semidirect product if and only if it is an ample internal semidirect product of $(\hat{\X},\hat{\Delta},\hat{S}_\X)$ with $(\hat{\N},\hat{\Gamma},\hat{T})$.
\end{itemize}
\end{lemma}

\begin{proof}
By Theorem~\ref{external.internal.par} the partial group $\L =  \X \ltimes_\varphi \N$ is the internal semidirect product of the partial subgroup $\hat{\X}=(\X,\One)$ with the partial subgroup $\hat{\N}=(\One,\N)$. Moreover, $(\One,n)^{(x,\One)}=(\One,n^{x^\phi})$ and $(x,\One)(\One,n)=(x,n)$ for all $x\in \X$ and $n\in\N$. In particular, $S = (S_\X, T) = (S_\X, \One) \cdot (\One,T)=\hat{S}_\X\hat{T}$. Note also that $\hat{\X}$ leaves $\hat{\Gamma}$ invariant, since $\Gamma$ is $\X$-invariant.

\smallskip

\noindent By Lemma~\ref{subgroups}(c), the maps $\alpha\colon \X\rightarrow \hat{\X},x\mapsto (x,\One)$ and $\beta\colon\N\rightarrow\hat{\N},n\mapsto (\One,n)$ are isomorphisms of partial groups. Observe also that $S_\X^\alpha=\hat{S}_\X$, $\Delta_\X^\alpha=\hat{\Delta}_\X$, $T^\beta=\hat{T}$ and $\Gamma^\beta=\hat{\Gamma}$. Therefore, the fact that the triples $(\hat{\X}, \hat{\Delta}_\X,\hat{S}_\X)$ and $(\hat{\N},\hat{\Gamma},\hat{T})$ form localities follows directly from the assumption that $(\X, \Delta_\X, S_\X)$ and $(\N, \Gamma,T)$ are localities. One sees easily that $S\cap \hat{\X}=\hat{S}_\X$ and $S\cap \hat{\N}=\hat{T}$. Hence, $(\hat{\X},\hat{\Delta},\hat{S}_\X)$ and $(\hat{\N},\hat{\Gamma},\hat{T})$ are sublocalities of $(\L,\Delta^*,S)$.

\smallskip
 
\noindent Notice that $\C_{\hat{\X}}(\hat{\N}) = (\C_\X^\phi(\N), \One)$. Thus for every $(Q,\One) \in \hat{\Delta}_\X$ we have 
\[(Q,\One) \cap \C_{\hat{\X}}(\hat{\N}) = (Q,\One) \cap  (\C_\X^\phi(\N), \One) = (Q \cap \C_\X^\phi(\N), \One).\]
As $(\X,\Delta_\X,S_\X)$ acts on $(\N,\Gamma,T)$, we have $Q\cap \C_\X^\phi(\N)\in\Delta_\X$. So we conclude that $(Q,\One) \cap \C_{\hat{\X}}(\hat{\N}) \in \hat{\Delta}_\X$ whenever $(Q, \One) \in \hat{\Delta}_\X$. Thus, all assumptions in Hypothesis \ref{H:ProdLocInternal} are satisfied with $(\hat{\X}, \hat{\Delta}_\X,\hat{S}_\X)$ and $(\hat{\N},\hat{\Gamma},\hat{T})$ in place of $(\X,\Delta_\X,S_\X)$ and $(\N,\Gamma,T)$. Adopt the notation introduced there accordingly, so that in particular $\Delta$ and $\Delta^+$ are defined. Furthermore, define $\Delta_\phi$ and $\Delta^+_\phi$ as in Definition~\ref{D:ExtProdLoc}. Then 
\begin{eqnarray*}
\Delta&=&\{P\leq S\colon \exists (Q,\One)\in\hat{\Delta}_\X\;\exists(\One,R)\in\hat{\Gamma}\mbox{ such that }(Q,\One)(\One,R)\leq P\}\\
&=& \{P\leq S\colon \exists Q\in\Delta_\X\;\exists R\in\Gamma\mbox{ such that }(Q,R)\leq P\}=\Delta_\phi.
\end{eqnarray*}
We argue next that $\Delta^+=\Delta^+_\phi$. Observe first that, for every $P \leq S$, we have 
\begin{eqnarray*}
P_{\hat{\X}} &=& \{ (s,\One) \in (\C_{S_\X}^\phi(\N), \One) \colon \exists (\One,t) \in \hat{T} \mbox{ such that } (s,\One)(\One,t) \in P\}\\
&=& \{(s,\One)\colon s\in \C_{S_\X}^\phi(\N),\;\exists t\in T\mbox{ such that }(s,t)\in P\}\\
&=& \{(s,\One)\colon s\in P_\X^\phi\}=(P_\X^\phi,\One)
\end{eqnarray*}
and
\begin{eqnarray*}
P_{\hat{\N}}&=&\{(\One,t)\in (\One,T) \colon \exists (s,\One) \in (\C_{S_\X}^\phi(\N), \One) \mbox{ such that } (s,t) \in P\}\\
&=&\{(\One,t)\colon t\in T,\;\exists s\in \C_{S_\X}^\phi(\N)\mbox{ such that }(s,t)\in P\}\\
&=&\{(\One,t)\colon t\in P_\N^\phi\}=(\One,P_\N^\phi).
\end{eqnarray*}
Hence, 
\begin{eqnarray*}
\Delta^+&=&\{P\leq S\colon P_{\hat{\X}}\in\hat{\Delta}_\X\mbox{ and }P_{\hat{\N}}\in\hat{\Gamma}\}\\
&=& \{P\leq S\colon (P_\X^\phi,\One)\in\hat{\Delta}_\X\mbox{ and }(\One,P_\N^\phi)\in\hat{\Gamma}\}\\
&=& \{P\leq S\colon P_\X^\phi\in\Delta_\X\mbox{ and }P_\N^\phi\in\Gamma\}=\Delta^+_\phi.
\end{eqnarray*}
So $\Delta_\phi=\Delta$ and $\Delta^+_\phi=\Delta^+$. Hence, (a) follows from Lemma~\ref{L:ProdLocInternal}. Moreover,  $\Delta=\Delta_\phi\subseteq\Delta^*\subseteq\Delta^+_\phi=\Delta^+$ and $\Delta^*$ is closed under taking $\L$-conjugates and overgroups in $S$. So using Corollary \ref{C:ProdLocInternal}, we conclude that $(\L, \Delta^*, S)$ is a locality and (b) holds. Moreover, as all the assumptions in Definition~\ref{D:InternalProdLoc} are fulfilled, we deduce that (c) holds. Now (d) holds clearly as well.
\end{proof}

\subsection{Semidirect products of groups with localities}\label{Ss:SemidirectProdGroupLoc}

In this subsection we consider internal and external semidirect products of a finite group $\X$ with a locality $(\N,\Gamma,T)$. Note that we can attach to any finite group $\X$ a locality $(\X,S_\X,\Delta_\X)$ by choosing a Sylow $p$-subgroup $S_\X$ of $\X$ and taking $\Delta_\X$ to be the set of all subgroups of $S_\X$. Hence, the concepts and results we will state here are actually special cases of the ones introduced in the previous subsection. We feel however that it is worth spelling out this special case, in particular since it is needed to define wreath products of localities.

\begin{definition}[Internal semidirect product of a group with a locality]\label{D:InternalGroupLoc}
Let $(\L,\Delta^*,S)$ be locality. Then we say that $(\L,\Delta^*,S)$ is an internal semidirect product of a subgroup $\X$ with a sublocality $(\N,\Gamma,T)$ if, setting  $S_\X:=S\cap\X$ and writing $\Delta_\X$ for the set of all subgroups of $S_\X$, the subgroup $S_\X$ is a Sylow $p$-subgroup of $\X$ and $(\L,\Delta^*,S)$ is the internal semidirect product of the sublocality $(\X,\Delta_\X,S_\X)$ with $(\N,\Gamma,T)$. We say that the internal semidirect product of $\X$ with $(\N,\Gamma,T)$ is sparse (or ample) if the internal semidirect product of $(\X,\Delta_\X,S_\X)$ with $(\N,\Gamma,T)$ is sparse (or ample, respectively).
\end{definition}

\begin{remark}
Let $(\L,\Delta^*,S)$ be a locality which is an internal semidirect product of a subgroup $\X$ with a locality $(\N,\Gamma,T)$. Set $S_\X:=S\cap\X$, write $\Delta_\X$ for the set of all subgroups of $S_\X$, and adopt the notation introduced in Hypothesis~\ref{H:ProdLocInternal} (which is also used in Definition~\ref{D:InternalProdLoc}). Then, as the trivial subgroup is an element of $\Delta_\X$ and $\Gamma$ is closed under taking overgroups in $T$, we have 
\[\Delta=\{P\leq S\colon \exists R\in \Gamma\mbox{ such that }R\leq P\}=\{P\leq S\colon P\cap T\in\Gamma\}\]
and
\[\Delta^+=\{P\leq S\colon P_\N\in\Gamma\}.\]
In particular, $(\L,\Delta^*,S)$ is a sparse internal semidirect product of $\X$ with $(\N,\Gamma,T)$ if $\Delta^*$ is the set of all overgroups of the elements of $\Gamma$.
\end{remark}

\begin{definition}[Action of a group on a locality]
 Let $\X$ be a group and $(\N,\Gamma,T)$ a locality. Then $\X$ acts on $(\N,\Gamma,T)$ if there exists a group homomorphism $\phi\colon \X\rightarrow \Aut(\N)$ such that $\Gamma$ is $\X$-invariant, i.e. $R^{x^\phi}\in\Gamma$ for all $R\in\Gamma$ and $x\in\X$. 
\end{definition}

\begin{remark}\label{R:ActionGroupLoc}
 Suppose a finite group $\X$ acts on a locality $(\N,\Gamma,T)$. Let $S_\X$ be a Sylow $p$-subgroup of $\X$ (where $T$ is a $p$-group), and let $\Delta_\X$ be the set of all subgroups of $S_\X$. Then $(\X,\Delta_\X,S_\X)$ acts on $(\N,\Gamma,T)$, as $\Gamma$ is $\X$-invariant and $\Delta_\X$ contains every subgroup of $S_\X$. Moreover, setting $\L=\X\ltimes_\phi\N$ and $S=(S_\X,T)$, and using the notation introduced in Definition~\ref{D:ExtProdLoc} we have
\[\Delta_\phi=\{P\leq S\colon \exists R\in\Gamma\mbox{ such that }(\One,R)\leq P\}.\]
and
\[\Delta^+_\phi=\{P\leq S\colon P_\N^\phi\in \Gamma\}.\]
Here
\[P_\N^\phi=\{t\in T\colon \exists s\in \C_{S_\X}^\phi(\N)\mbox{ such that }(s,t)\in P\}.\]
\end{remark}

\begin{definition}[External semidirect product of a group with a locality]\label{D:ExternalGroupLoc}
Suppose a finite group $\X$ acts on a locality $(\N,\Gamma,T)$ via $\phi$ . Pick a Sylow $p$-subgroup $S_\X$ of $\X$ and write $\Delta_\X$ for the set of all subgroups of $S_\X$. Then an external semidirect product of $\X$ with $(\N,\Gamma,T)$ via $\phi$ (or of $(\X,S_\X)$ with $(\N,\Gamma,T)$ via $\phi$) is a triple $(\L,\Delta^*,S)$ which is an external semidirect product of $(\X,\Delta_\X,S_\X)$ with $(\N,\Gamma,T)$.
The external semidirect product of $\X$ with $(\N,\Gamma,T)$ is called sparse (or ample) if the external semidirect product of $(\X,S_\X,\Delta_\X)$ with $(\N,\Gamma,T)$ is sparse (or ample). 
\end{definition}

\begin{remark}
Using Remark~\ref{R:ActionGroupLoc}, we see that an external semidirect product of a group $\X$ with a locality $(\N,\Gamma,T)$ (via a homomorphism $\phi\colon\X\rightarrow\Aut(\N)$) is a triple $(\L,\Delta^*,S)$ such that the following hold:
\begin{itemize}
 \item [(1)] $\L=\X\ltimes_\phi\N$ and $S=(S_\X,T)$ for some Sylow $p$-subgroup $S_\X$ of $\X$;
 \item [(2)] $(\One,R)\in\Delta^*$ for every $R\in\Gamma$. 
 \item [(3)] If $P\in\Delta^*$, then $P_\N^\phi\in \Gamma$.
 \item [(4)] $\Delta^*$ is closed under taking $\L$-conjugates and overgroups in $S$. 
\end{itemize}
Moreover, a sparse external semidirect product of $\X$ with $(\N,\Gamma,T)$ is a triple $(\L,\Delta^*,S)$ such that (1) holds and $\Delta^*$ is the set of all overgroups in $S$ of subgroups of the form $(\One,R)$ with $R\in\Gamma$. 
Similarly, an ample external semidirect product of $\X$ with $(\N,\Gamma,T)$ is a triple $(\L,\Delta^*,S)$ such that (1) holds and $\Delta^*$ is the set of all $P\leq S$ with $P_\N^\phi\in\Gamma$.
\end{remark}

\section{Wreath products}\label{S:WreathProducts}

\begin{definition}[External direct product]\label{direct}
Let $n\in\N$ such that for every $1\leq i\leq n$, we are given localities $(\L_i, \Delta_i, S_i)$ with partial products $\Pi_i \colon \D_i \rightarrow \L_i$. The external direct product of the localities $\L_i$, is the locality $(\L = \L_1 \times \dots \times \L_n, \Delta_1 * \dots * \Delta_n, S = S_1 \times \dots \times S_n)$ where
\begin{enumerate}
\item $\L = \{ (f_1, \dots, f_n) \mid f_i \in \L_i \}$;
\item for every word $w = ( (f_{1, 1}, \dots, f_{1,n}), (f_{2,1}, \dots, f_{2,_n}), \dots, (f_{m,1}, \dots, f_{m,n}))$ write $w_i = (f_{1,i}, \dots f_{m,i}) \in \L_i$ and set 
\[ \D = \D(\L) = \{ w \in \L \mid w_i \in \D_i \text{ for every } 1\leq  i \leq n \};\]
\item the partial product $\Pi \colon \D \rightarrow \L$ is given by 
$ \Pi(w) =(\Pi_1(w_1), \Pi_2(w_2), \dots, \Pi_n(w_n))$;
\item the inversion map $^{-1} \colon \L \rightarrow \L$ is given by $(f_1, \dots, f_n)^{-1} = (f_1^{-1}, \dots, f_n^{-1})$;
\item $\Delta_1 * \dots * \Delta_n = \{ P \leq S \mid  Q_1 \times Q_2 \times \dots \times Q_n \leq P \text{ for some } Q_i  \in \Delta_i \}$.
\end{enumerate}
\end{definition}

\begin{remark}
The fact that $(\L = \L_1 \times \dots \times \L_n, \Delta_1 * \dots * \Delta_n, S = S_1 \times \dots \times S_n)$ is indeed a locality can be proved with an argument similar to the one used to prove the statement for $n=2$ in \cite[Lemma 5.1]{direct.prod}. 
\end{remark}

\begin{definition} If $(\L, \Delta, S)$ is a locality and $k\geq 1$ is an integer, we write $(\L^k, \Delta^{*k}, S^k)$ to denote the external direct product of $k$ copies of $(\L,\Delta, S)$.
\end{definition}

\begin{lemma}\label{L:WreathProdAction}
Let  $k\geq 2$ be an integer, let $X$ be a subgroup of the symmetric group $\Sym(k)$ and let $(\N, \Gamma, T)$ be a non-trivial locality. Then  $X$ acts on $(\N^k, \Gamma^{*k}, T^k)$ with natural action respecting $\Gamma^{*k}$.
\end{lemma}

\begin{proof}
Write $\N_i$ for the $i$-th copy of $\N$ appearing in the direct product $\N^k$.
Then every element of $\N^k$ is of the form $(n_1, \dots, n_k)$ for some $n_i \in \N_i$ and the natural action of $X$ on the indices $i$ gives a group homomorphism
\[ X \rightarrow \Aut(\N^k), \quad x \mapsto ( (n_1, \dots, n_k) \mapsto (n_1, \dots, n_k)^x = (n_{1^x}, \dots, n_{k^x}) ).\]
It remains to show that the action of $X$ preserves $\Gamma^{*k}$. Suppose $P \in \Gamma^{*k}$. Then by Definition \ref{direct}(5) for every $1\leq i\leq k$ there exists $Q_i \in \Gamma$ such that $Q_1 \times \dots \times Q_k \leq P$. Then for every $x \in X$ we have $Q_{1^x} \times \dots \times Q_{k^x} \leq P^x$. Since $Q_{i^x} = Q_j$ for some $1\leq j\leq k$ and $ Q_j \in \Gamma$, we deduce that $P^x \in \Gamma^{*k}$. Therefore $X$ acts on $\Gamma^{*k}$ and so it acts on the locality  $(\N^k, \Gamma^{*k}, T^k)$.
\end{proof}

\begin{definition}
Let  $k\geq 2$ be an integer and let $X$ be a subgroup of the symmetric group $\Sym(k)$.
\begin{enumerate}
\item A wreath product of $X$ with a locality $(\N,\Gamma,T)$ is an external semidirect product of $X$ with $\N^k$ via the action defined in Lemma~\ref{L:WreathProdAction};
\item The sparse (or ample) wreath product of $X$ with the locality $(\N, \Gamma, T)$ is the external semidirect product of $X$ with $(\N^k, \Gamma^{*k}, T^k)$ which is sparse (or ample, respectively).
\end{enumerate}
\end{definition}

\begin{remark}
If $(\L, \Delta, S)$ is a wreath product of $X$ with $(\N, \Gamma, T)$, then by Theorem \ref{Ext.loc} we can consider $\L$ as an internal semidirect product of a canonical image of $X$ in $\L$ acting on a canonical image of $(\N^k, \Gamma^{*k}, T^k)$ in $\L$. We will use this fact without further reference.
\end{remark}

\begin{lemma}\label{centralizer}
Let $k\geq 2$ be an integer and let $(\L, \Delta, S)$ be a wreath product of the group $X \leq \Sym(k)$ with $(\N, \Gamma, T)$. 
Suppose $\H \leq \N^k$ is such that for every $1 \leq  i \leq k$ there exists an element $h_i\in \H$ that has all entries but the $i$-th one equal to $1$. Then $\C_\L(\H) = \C_{\N^k}(\H)$. In particular $\C_\L(\N^k) \leq \N^k$.
\end{lemma}  

\begin{proof}
We identify $\L$ with the standard internal semidirect product of $X$ acting on $(\N^k, \Gamma^{*k}, T^k)$.
Let $g \in \C_\L(\H)$. Then $g=\Pi(x,m)$ for some $x \in X$ and $m\in \N^k$ and $S_g = S_{(x,m)}$.
Aiming for a contradiction, suppose $x\neq 1$. So there exists $1\leq i \leq k$ such that $i^x\neq i$. Then $h_{i^x}$ has the $i$-th entry equal to $1$, and the same holds for $(h_{i^x})^m=(h_i)^g$. Since the $i$-th entry of $h_i$ is not equal to $1$ by assumption, we deduce that $(h_i)^g \neq h_i$, contradicting the fact that $g \in \C_\L(\H)$. Therefore $x=1$ and $g=m\in \C_{\N^k}(\H)$. Since this is true for every $g \in  \C_\L(\H)$, we deduce that $\C_\L(\H) =\C_{\N^k}(\H)$.

By definition of wreath product, the locality $\N$ is non-trivial. Hence we can apply the lemma with $\N^k$ in place of $\H$ and we conclude that $\C_\L(\N^k) \leq \N^k$.
\end{proof}

\begin{lemma}\label{centralizer.1notGamma}
Let $k\geq 2$ be an integer and let $(\L, \Delta, S)$ be a wreath product of the group $X \leq \Sym(k)$ with $(\N, \Gamma, T)$. 
If  $1\notin \Gamma$, then for every $P \in \Delta$ and every $P\leq R \leq S$ we have $\C_\L(R) = \C_{\N^k}(R)$.
\end{lemma}

\begin{proof}
 Suppose $P \in \Delta$ and $P\leq R \leq S$. Then for every $1\leq i \leq k$  there exists $Q_i \in \Gamma$ such that $Q_1 \times \dots \times Q_k \leq P$.
By assumption $Q_i \neq 1$ for every $1 \leq i \leq k$. Hence by Lemma \ref{centralizer} we get

\[\C_{\L}(R) \leq \C_{\L}(P) \leq \C_{\L}(Q_1 \times \dots \times Q_k) = \C_{\N^k}(Q_1 \times \dots \times Q_k) \leq \N^k.\]

Hence $\C_\L(R) = \C_{\N^k}(R)$. 
\end{proof}

\bibliographystyle{alpha}
\bibliography{my.books}

\begin{thebibliography}{BLO03}

\bibitem[BLO03]{BLO2}
C.~Broto, R.~Levi, and B.~Oliver.
\newblock {The homotopy theory of fusion systems}.
\newblock 16:779--856, 2003.

\bibitem[Che]{loc1}
A.~Chermak.
\newblock {Finite Localities I}.
\newblock {\em preprint: arXiv:1505.07786v3}.

\bibitem[Che13]{Ch}
A.~Chermak.
\newblock Fusion systems and localities.
\newblock {\em Acta Math. 211}, pages 47--139, 2013.

\bibitem[Hen15]{Henke:2015a}
E.~Henke.
\newblock Products of partial normal subgroups.
\newblock {\em Pacific J. Math.}, 279(1-2):255--268, 2015.

\bibitem[Hen17]{direct.prod}
E.~Henke.
\newblock Direct and central product of localities.
\newblock {\em J. Algebra}, 491:158--189, 2017.

\bibitem[Hen19]{subcentric}
E.~Henke.
\newblock Subcentric linking systems.
\newblock {\em Trans. Amer. Math. Soc.}, 371, 2019.

\end{thebibliography}

\end{document}